\documentclass[pdflatex,sn-mathphys-num]{sn-jnl}

\usepackage{graphicx}%
\usepackage{multirow}%
\usepackage{amsmath,amssymb,amsfonts}%
\usepackage{amsthm}%
\usepackage[title]{appendix}%
\usepackage{xcolor}%
\usepackage{textcomp}%
\usepackage{manyfoot}%
\usepackage{booktabs}%
\usepackage{algorithm}%
\usepackage{algorithmicx}%
\usepackage{algpseudocode}%
\usepackage{listings}%

\theoremstyle{thmstyleone}%
\newtheorem{theorem}{Theorem}
\newtheorem{remark}[theorem]{Remark}%
\newtheorem{corollary}[theorem]{Corollary}
\newtheorem{lemma}[theorem]{Lemma}
\theoremstyle{thmstyletwo}%

\theoremstyle{thmstylethree}%
\begin{document}

\title[Article Title]{New perspectives on operator radius bounds in $A-$weighted frameworks}


\author[1]{\fnm{Bikram} \sur{Das}}\email{dasbikram642@gmail.com}

\author*[2]{\fnm{Chandal} \sur{ Nahak}}\email{cnahak@maths.iitkgp.ac.in}
\equalcont{The author contributed equally to this work.}

\affil[1,2]{\orgdiv{Department of Mathematics}, \orgname{Indian Institute of Technology Kharagpur}, \orgaddress{\city{Kharagpur}, \postcode{721302}, \state{West Bengal}, \country{India}}}


\abstract{By employing the Moore$-$Penrose inverse of a bounded linear operator, we derive several bounds for the numerical radius and operator norms of the sum of operators in semi$-$Hilbertian space that generalize and improve the classical bounds. We establish novel inequalities pertaining to the $\mathbb{A}$$-$Davis$-$Wielandt radius for $n \times n$ operator matrices and further explore their ramifications, particularly concerning $\mathbb{A}$$-$Davis$-$Wielandt radius bounds for $2 \times 2$ operator matrices, where diagonal operator matrix $\mathbb{A}$ contains positive bounded operator $A.$ Ultimately, we get an improved upper bound for the $A$$-$numerical radius inequalities relating to the commutators of operators. }




\maketitle

\section{Introduction}\label{sec1}

A big part of modern functional analysis is the study of operator inequalities, especially when it comes to understanding the geometry of operators working on Hilbert and semi-Hilbertian spaces. Recently, a lot of work has gone into making better bounds for different types of operator radii, such as the numerical radius, the Davis–Wielandt radius, and their more generalized versions. These quantities not only encapsulate profound spectral information but also function as essential instruments in matrix analysis, perturbation theory, and quantum information science.

The Moore–Penrose inverse of bounded linear operators has created new opportunities for achieving tighter inequalities in the context of semi-Hilbertian spaces created by a positive bounded operator $A.$ Using this method, it is possible to extend the conclusions of classical Hilbert space to a more broad context, where the geometry is governed by the semi-inner product $\langle\cdot,\cdot\rangle_{A}.$

Let $\mathcal{H}$ be a complex Hilbert space endowed with an inner product denoted by $\langle \cdot , \cdot \rangle$ and the associated norm $\| \cdot \|$. An element $z \in \mathcal{H}$ such that $\|z\| = 1$ is a unit vector in the Hilbert space $\mathcal{H}$.  Let $\mathbb{B} (\mathcal{H})$ denote the set of all bounded linear operators on the Hilbert space $\mathcal{H}$. Let $\mathbb{B} (\mathcal{H})^+$ be the collection of all positive operators of $\mathbb{B} (\mathcal{H}),$  that is, 
$\mathbb{B} (\mathcal{H})^+=\{S\in {\mathbb{B} (\mathcal{H})}: \langle Sz,z \rangle \ge 0, \hspace{0.2cm} \forall z\in {\mathcal{H}}\}.$ 
For any $S\in{\mathbb{B} (\mathcal{H})}$ we denote by $\mathcal{R}(S),\mathcal{N}(S)$ and $S^*,$ the range of $S$, the null space of $S$ and the adjoint of $S,$ respectively.

For $S\in{\mathbb{B}(\mathcal{H}) } $, we denote the operator norm of $S$ by  $\|S\|$, that is,  $$\| S\|= \sup\limits_{\substack{z\in{\mathcal{H}} \\ \|z\|=1}}\|Sz\|=\sup\limits_{\substack{z,y\in{\mathcal{H}} \\\|z\|=\|y\|=1}}|\langle S z,y \rangle|.$$
The numerical range of $S\in{\mathbb{B} (\mathcal{H})}$, represented as $W(S)$, is defined by $W(S)=\big\{\langle Sz, z \rangle : z \in {\mathcal{H}}, \|z\|=1\big\}.$ 
The numerical radius of $S\in{\mathbb{B} (\mathcal{H})}$, which is represented by the symbol $\omega(S)$, is defined as $$\omega(S)=\sup\limits_{\substack{z\in{\mathcal{H}} \\\|z\|=1}}\big| \langle Sz,z \rangle\big|.$$

For any $A\in{\mathbb{B} (\mathcal{H})^+},$ it induces a positive semidefinite sesquilinear form : 
$$\langle.,.\rangle_{A}:\mathcal{H}\times\mathcal{H}\to \mathbb{C},\hspace{0.5cm}\langle z,y \rangle_{A}=\langle Az,y \rangle.$$
The semi-inner product $\langle .,. \rangle_{A}$ induces the seminorm $\|.\|_{A},$ that is, $\|z\|_{A}=\sqrt{\langle z,z \rangle _{A}}=\|A^{\frac{1}{2}}z\|$ for all $z \in{\mathcal{H}}.$ Here $A^{\frac{1}{2}}$ is the square root of $A.$ It is seen that $\|z\|_{A}=0$ if and only if $A$ is injective. Further, $(\mathcal{H}, \|.\|_{A})$
 is complete if and only if $\mathcal{R}(A)$ is closed in $\mathcal{H}.$ The $A-$ operator  seminorm of $S\in{\mathbb{B}(\mathcal{H})}$ is defined by $\|S\|_{A}=\sup\{\|Sz\|_{A}:\|z\|_{A}=1\}.$
It should be mentioned here that $\|S\|_{A}=+\infty$ for some  $S\in{\mathbb{B}(\mathcal{H})}.$ We define $\mathbb{B}^{A}(\mathcal{H})$ by $$\mathbb{B}^{A}(\mathcal{H})=\{S\in{\mathbb{B}(\mathcal{H})}:\|S\|_{A}<\infty\}.$$ It is well-known that $\mathbb{B}^{A}(\mathcal{H})$ is not a subalgebra of $\mathbb{B}(\mathcal{H})$ in general and $\|S\|_{A}=0$ if and only if $S^*AS=0.$
An operator $R\in{\mathbb{B}(\mathcal{H})}$ is said to be an $A-$adjoint operator of $S\in{\mathbb{B}(\mathcal{H})}$ if $\langle Sz,y \rangle_{A}=\langle z,Ry\rangle_{A}$ for every $z,y \in{\mathcal{H}},$ that is, $AR=T^*A.$

In general, the existence and uniqueness of an $A-$adjoint operator are not guaranteed. The set of all operators that admit $A-$adjoints is denoted by $\mathbb{B}_{A}(\mathcal{H}).$ 
Also, we denote by $\mathbb{B}_{A^{\frac{1}{2}}}(\mathcal{H}),$ the set of all operators that admit $A^{\frac{1}{2}}-$adjoints. 
The following inclusions hold $$\mathbb{B}_{A}(\mathcal{H})\subseteq \mathbb{B}_{A^{\frac{1}{2}}}(\mathcal{H})\subseteq\mathbb{B}^{A}(\mathcal{H})\subseteq\mathbb{B}(\mathcal{H})$$ with equality if and only if $A$ is one-to-one and has a closed range. We use the notation $\mathbb{ CR} (\mathcal{H})$
for all operators having closed ranges. For
$S\in{\mathbb{ CR} (\mathcal{H})},$ the Moore-Penrose inverse of $S$ is the unique linear mapping $S^\dag:\mathcal{R}(A)\bigoplus\mathcal{R}(A)^{\perp}\to\mathcal{H},$ which follows the four equations:

$(1) \hspace{0.1cm} SS^{\dag}S=S,$ $(2)\hspace{0.1cm} S^{\dag}SS^{\dag}=S^{\dag},$  $(3) \hspace{0.1cm}(SS^{\dag})^*=SS^{\dag},$  $(4) \hspace{0.1cm} (S^{\dag}S)^*=S^{\dag}S$. In general, $S^{\dag}\notin \mathbb{B}(\mathcal{H})$ and $S^{\dag} \in \mathbb{B}(\mathcal{H})$ if and only if $S \in \mathbb{CR}(\mathcal{H}).$ This yields that each matrix has the Moore-Penrose inverse because in case of finite dimensional Hilbert space $\mathcal{H},$ every operator in $\mathbb{B}(\mathcal{H})$ has closed range. One can easily see that:

$(a)\hspace{0.1cm}S=SS^{\dag}S=(SS^{\dag})^{*}S=(S^{\dag})^{*}|S|^2$ and 
$(b)\hspace{0.1cm}S=SS^{\dag}S=S(S^{\dag}S)^{*}=|S^*|^2(S^{\dag})^{*},$ where $|S|=(S^*S)^{\frac{1}{2}}$ and $|S^*|=(SS^*)^{\frac{1}{2}}$. For $S\in \mathbb{B}_{A}(\mathcal{H}),$ the solution of the equation $AY=S^{*}A$ is an $A-$ adjoint operator of $S,$ which is denoted by $S^{*_A}.$ 

Observe that, $S^{*_A}=A^{\dag}S^{*}A$ and the $A-$ adjoint operator $S^{*_A}$ satisfies $AS^{*_A}=S^{*_A}A,$ $\mathcal{R}(S^{*_A})\subseteq \overline{\mathcal{R}(A)}$ and $\mathcal{N}(S^{*_A})=\mathcal{N}(S^{*}A).$ If $S_1,S_2\in {\mathbb{B}_{A}(\mathcal{H})}$ then $\|S_{1}S_{2}\|_{A}\le\|S_{1}\|_{A}\|S_{2}\|_{A}$ and $(S_{1}S_{2})^{*_A}=S_{2}^{*_A}S_{1}^{*_A}$.

For $S\in {\mathbb{B}_{A}(\mathcal{H})},$ we define:
$$\|S\|_{A}=\sup\{|\langle Sz,y \rangle _{A}|: z,y \in {\mathcal{H}}, \|z\|_{A}=\|y\|_{A}=1\}.$$ An operator $S\in{\mathbb{B}(\mathcal{H})}$ is said to be $A-$positive if $AS\in{\mathbb{B}(\mathcal{H})^{+}}.$ Note that if $S$ is $A-$positive, then $\|S\|_{A}=\sup\{\langle Sz,z \rangle_{A}: z\in{\mathcal{H}}, \|z\|_{A}=1\}.$ An operator $S\in{\mathbb{B}(\mathcal{H})}$ is called $A-$selfadjoint if $AS$ is selfadjoint, that is, $AS=S^*A.$  

It is clear that if $S$ is $A-$selfadjoint, then $S\in{\mathbb{B}_{A}(\mathcal{H})}.$ In general, $S=S^{*_A}$ is not always true. Additionally, if $S\in \mathbb{B}_{A}(\mathcal{H}),$ then  $S=S^{*_A}$ if and only if $S$ is $A-$selfadjoint and $\mathcal{R}(S)\subseteq \overline{\mathcal{R}(A)}.$ It is obvious that $S^{*_A}S$ and $SS^{*_A}$ are $A-$
selfadjoint and $A-$positive for $S\in {\mathbb{B}_{A}(\mathcal{H})}$ and $\|S\|^2_{A}=\|S^{*_A}\|^2_{A}=\|SS^{*_A}\|_{A}=\|S^{*_A}S\|_{A}.$

The $A-$numerical radius and the $A-$crawford number of  $S\in{\mathbb{B}(\mathcal{H})}$ are defined by 

\begin{equation*}
    \omega_{A}(S)=\sup\{|\langle Sz,z \rangle_{A}|:z\in{\mathcal{H}},\|z\|_{A}=1\}\end{equation*} and 
\begin{equation*}
    c_{A}(S)=\inf\{|\langle Sz,z \rangle _{A}|:z\in{\mathcal{H}},\|z\|_{A}=1\}.\end{equation*}

The $A-$spectral radius of an operator $S\in {\mathbb{B}_{A}(\mathcal{H})}$ is defined by 
$r_{A}(S)=\lim\limits_{\substack{n\to\infty}}\|S^n\|^{\frac{1}{n}}_{A}.$
Observe that the seminorm $\omega_{A}(.)$ on $\mathbb{B}_{A^{\frac{1}{2}}}(\mathcal{H})$ is equivalent to $\|.\|_{A}$ by the following inequality which is given in [\citenum{baklouti2018joint}]:
\begin{equation}\label{(1)}
    \frac{1}{2}\|S\|_{A}\le\omega_{A}(S)\le\|S\|_{A},\tag{1}\end{equation} where $S\in{\mathbb{B}_{A^{\frac{1}{2}}}(\mathcal{H})}.$
    In 2020, Feki proved in [\citenum{feki2020spectral}, Theorem 3] that if $S\in{\mathbb{B}_{A^{\frac{1}{2}}}(\mathcal{H})},$ then
\begin{equation}\label{2}
  r_{A}(S)\le \omega_{A}(S).  
\tag{2}\end{equation}
Further, in [\citenum{feki2020spectral}, Theorem 7], it is proved that  
\begin{equation}\label{(3)}
 r_{A}(S)\le\omega_{A}(S)\le\frac{1}{2}\left(\|S\|_{A}+\sqrt{\|S^2\|_{A}}\right).  
\tag{3}\end{equation}

It is well-known that if $S\in{\mathbb{B}_{A^{\frac{1}{2}}}(\mathcal{H})}$ is $A-$selfadjoint, then $ r_{A}(S)=\omega_{A}(S)=\|S\|_{A}$ and if $AS^2=0,$ then 
$ \omega_{A}(S)=\frac{1}{2}\|S\|_{A}.$

In [\citenum{bhunia2020inequalities}, Corollary 2.7, \citenum{zamani2019numerical}, Theorems 2.10, 2.12], it is demonstrated that if $S\in{\mathbb{B}_{A}(\mathcal{H})},$ then
\begin{equation}\label{(4)}
\frac{1}{4}\|S^{*_A}S+SS^{*_A}\|_{A}\le\omega^2_{A}(S)\le \frac{1}{2}\|S^{*_A}S+SS^{*_A}\|_{A}  
\tag{4}\end{equation}
Since $\omega_{A}(S^{*_A})=\omega_{A}(S),$ so we can rewrite the inequality (\ref{(4)}) as 
\begin{equation*}\frac{1}{4}\|S^{*_A}(S^{*_A})^{*_A}+(S^{*_A})^{*_A}S^{*_A}\|_{A}\le\omega^2_{A}(S)\le\frac{1}{2}\|S^{*_A}(S^{*_A})^{*_A}+(S^{*_A})^{*_A}S^{*_A}\|_{A}.
\end{equation*}
Recently, Pintu [\citenum{bhunia2025improved}, Corollary 2.1] devloped some norm inequalities for the sum of two operators by using the concept of Moore-Penrose inverse of bounded linear operator with closed range:
\begin{equation}\label{(5)}
\| T_1+T_2\|\le \sqrt{\|T_1^{*}T_1+T_2^{\dag}T_2\|\hspace{0.07cm}\|T_2T_2^{*}+T_1T_1^{\dag}\|}    
\tag{5}\end{equation}
and
\begin{equation}\label{(6)}
\hspace{1.2cm}\| T_1+T_2\|\le \sqrt{\|T_2^{*}T_2+T_1^{\dag}T_1\|\hspace{0.07cm}\|T_1T_1^{*}+T_2T_2^{\dag}\|}    
\hspace{1cm} \forall T_1,T_2\in{\mathbb{CR}(\mathcal{H})}.\tag{6}\end{equation}
Further, in [\citenum{bhunia2025improved}, Corollary 2.2] it is established that
\begin{equation}\label{(7)}
\omega(T_1+T_2)\le \frac{1}{2}\hspace{0.08cm}\|T_1^{*}T_1+T_2^{\dag}T_2+T_2T_2^{*}+T_1T_1^{\dag}\|    
\tag{7}\end{equation}
and
\begin{equation}\label{(8)}
\hspace{1.2cm}\omega(T_1+T_2)\le \frac{1}{2}\hspace{0.07cm}\|T_2^{*}T_2+T_1^{\dag}T_1+T_1T_1^{*}+T_2T_2^{\dag}\|   
\hspace{1cm} \forall T_1,T_2\in{\mathbb{CR}(\mathcal{H})}.\tag{8}\end{equation}
In 2022, Feki [\citenum{feki2022some}, Theorem 2.2] extended a well-established result which is obtained by Kittaneh [\citenum{kittaneh2006spectral}, Theorem 1] :
\begin{equation}\begin{aligned}\label{(9)}\mbox{\footnotesize $r_{A}(A_1B_1+A_2B_2)
\le\frac{\sqrt{(\|B_1A_1\|_{A}-\|B_2A_2\|_{A})^2+4\|B_1A_2\|_{A}\|B_2A_1\|_{A}}+\|B_1A_1\|_{A}+\|B_2A_2\|_{A}}{2},$}\end{aligned}\tag{9}\end{equation}
$\forall A_1,A_2,B_1,B_2\in{\mathbb{B}_{A}(\mathcal{H})}.$
By using the basic result 
\begin{equation}\label{(10)}
|a+b|\le2\int_0^1|sa+(1-s)b|ds\le |a|+|b|\hspace{0.4cm} \forall a,b\in{\mathbb{C}},
\tag{10}\end{equation}
an improvement of the well-known Cauchy-Schwarz inequality is obtained in  [\citenum{sababheh2024operator}, Corollary 2.1] by :

\begin{equation}\label{(11)}
    |\langle z,y \rangle|\le \left(\int_{0}^{1}|se^{i\theta}+(1-s)e^{-i\theta}|ds\right)\hspace{0.05cm}\|z\|\hspace{0.05cm}\|y\|\le\|z\|\hspace{0.05cm}\|y\|,\tag{11}\end{equation}
where
the angle between the non-zero vectors $z$ and $y$ can be defined by $\theta=\angle_{\ll z,y\gg}=\cos^{-1}\left(\frac{|\langle z,y \rangle|}{\|z\|\hspace{0.05cm}\|y\|}\right).$ 
We consider $\lambda(\theta)=\int_{0}^{1}|se^{i\theta}+(1-s)e^{-i\theta}|ds.$

By the calculation, we get $\lambda(\theta)=\frac{1}{4}\left(2+\cos \theta\hspace{0.05cm}\cot \theta\hspace{0.05cm} \log\frac{1+\sin \theta}{1-\sin \theta}\right),$ $\theta \ne n\pi,\hspace{0.04cm} n=0,1,2,\cdots$
We observe that 
\begin{equation*}
    \lambda(\theta)=\begin{cases}
[\frac{1}{2},1] & \text{for } \theta\ge0 \\
1 & \text{for } \theta=0
\end{cases}\end{equation*} and $\lambda(\theta+\pi)=\lambda(\theta).$

In the current year, Bakherad [\citenum{bakherad2025generalized}, Lemma 2.1] extended the inequality (\ref{(11)}):
\begin{equation}\label{(12)}
\mu_{t}(\theta_{ z,y })\hspace{0.07cm}\|z\|_{A}\hspace{0.07cm}\|y\|_{A} \le |\langle z,y \rangle_{A}|\le \delta(\theta_{ z,y })\hspace{0.07cm}\|z\|_{A}\hspace{0.07cm}\|y\|_{A},  
\tag{12}\end{equation}
where

$\theta_{z,y }=\angle_{\ll A^{\frac{1}{2}}z,A^{\frac{1}{2}}y \gg},$ $t\in [0,1],$ $r_{t}=\min\{t,1-t\},$ $\mu_{t}(\phi)=1-\frac{1}{2r_{t}}\Big(1-\sqrt{\cos^2\phi+(2t-1)^2\sin^2\phi}\Big)$ and $\lambda(\phi)=\frac{1}{4}\left(2+\cos \phi\hspace{0.05cm}\cot \phi\hspace{0.05cm} \log\Big(\frac{1+\sin \phi}{1-\sin \phi}\Big)\right).$

Because of the significance of numerical range and numerical radius, many generalizations have been discussed in the literature. One of these interesting generalizations is Davis–Wielandt shell which is described as $$DW(S)=\{(\langle Sz,z \rangle, \langle Sz,Sz \rangle ): z \in {\mathcal{H}}, \|z\|=1\},$$ for $S\in{\mathbb{B}(\mathcal{H})}.$

This concept was firstly mentioned by Davis [\citenum{davis1968shell}] and Wielandt [\citenum{wielandt1955eigenvalues}]. The projection of the set $DW(S)$ onto the first coordinate is evidently the classical numerical range $W(S)$. So, it offers additional information regarding $S$ and $W(S)$.

One can define the Davis-Wielandt radius of $S\in{\mathbb{B}(\mathcal{H})}$
by $$d\omega(S)=\sup\limits_{\substack{z\in{\mathcal{H}} \\ \|z\|=1}}\left\{\sqrt{|\langle Sz,z \rangle|^2+\|Sz\|^4}\right\}.$$

The $A-$ Davis-Wielandt radius of $S\in{\mathbb{B}_{A^{\frac{1}{2}}}(\mathcal{H})}$ is defined by 
$$d\omega_{A}(S)=\sup\limits_{\substack{z\in{\mathcal{H}} \\ \|z\|_{A}=1}}\left\{\sqrt{|\langle Sz,z \rangle_{A}|^2+\|Sz\|_{A}^4}\right\}.$$

\subsection{Organization of the paper}
The structure of this paper is as follows:

We present generalized and improved inequalities for the numerical radius and operator norms of the sum of operators in semi-Hilbertian space in Section \ref{sec2}. Additionally, this section leads to certain numerical radius inequalities and offers improvements on well-known operator norm  and numerical radius inequalities (\ref{(1)}), ((\ref{(4)}) - (\ref{(8)})).

In Section \ref{sec3}, we build new inequalities of $A-$Davis-Wielandt radius inequalities for $n\times n$ operator matrices and go deeper into their implications when the study turns to  $A-$Davis-Wielandt radius  bounds, specifically for $2\times 2$ operator matrices. When dealing with continuous functions and positive operators, this section helps to provide generalized $A-$numerical radius inequalities.

In Section \ref{sec4}, in order to find the $A-$ spectral radius bound for single operator in terms of its Moore-Penrose inverse, we prove a new inequality for the sum of products of operators, which is crucial. Finally, we obtain  refined upper bound for the  $A-$numerical radius inequalities of commutators of operators.

\section{Extension and improvement of inequalities (\ref{(1)}), ((\ref{(4)}) - (\ref{(8)}))}\label{sec2}
We begin this section with the following lemma which shows an inner product inequality relating the Moore-Penrose inverse of $S\in{\mathbb{CR}(\mathcal{H})}.$

\begin{lemma} \label{Lemma 2.1}
[\citenum{sababheh2024numerical}, Theorem 2.1]
Let $S\in{\mathbb{CR}(\mathcal{H})}.$ Then for any $z,y\in{\mathcal{H}},$
\begin{equation*}|\langle Sz,y \rangle|\le \sqrt{\langle S^{*}Sz,z \rangle\hspace{0.07cm}\langle SS^{\dag}y,y \rangle}.\end{equation*}
\end{lemma}

The following inequalities relate the $A-$inner product $\langle Sz,y \rangle_{A}$ according to the geometric angle formed by the appropriate $A-$weighted operator images of $z$ and $y.$ Sharper operator inequalities will be derived with the help of this geometric formulation. First, we prove the following lemma:

\begin{lemma} \label{Lemma 2.2}
Let  $S\in{\mathbb{B}_{A}(\mathcal{H})}.$ Then for any $z,y\in{\mathcal{H}},$
\begin{equation}\label{(13)}
|\langle Sz,y \rangle_{A}|\le\lambda(\theta_1)\|Sz\|_{A}\hspace{0.05cm}\|(SS^{\dag})^{*_A}y \|_{A}   
\tag{13}\end{equation}
and
\begin{equation}\label{(14)}
|\langle Sz,y \rangle_{A}|\le\lambda(\theta_2)\|S^{\dag}Sz \|_{A}\hspace{0.05cm}\|S^{*_A}y\|_{A},\tag{14}\end{equation}

where

\mbox{\footnotesize$\text{
$\theta_1=\angle_{\ll A^{\frac{1}{2}}Sz,A^{\frac{1}{2}}(SS^{\dag})^{*_A}y  \gg},$ $\theta_2=\angle_{\ll  A^{\frac{1}{2}}(S^{\dag}S)z,A^{\frac{1}{2}}S^{*_A}y \gg}$ 
and $\lambda(\phi)=\frac{2+\cos \phi\hspace{0.05cm}\cot \phi\hspace{0.05cm} \log\Big(\frac{1+\sin \phi}{1-\sin \phi}\Big)}{4}.$} $}

\end{lemma}

\begin{proof}
For any  $z,y\in{\mathcal{H}},$ we have
\begin{equation}
\begin{aligned}&
|\langle Sz,y \rangle_{A}|=|\langle (SS^{\dag})Sz,y \rangle_{A}|\\&\hspace{1.4cm}= |\langle Sz,(SS^{\dag})^{*_A}y \rangle_{A}|  \\&\hspace{1.4cm}= |\langle ASz,(SS^{\dag})^{*_A}y \rangle| \\&\hspace{1.4cm}= |\langle A^{\frac{1}{2}}Sz,A^{\frac{1}{2}}(SS^{\dag})^{*_A}y \rangle| \\&\hspace{1.4cm}\le \lambda(\theta_1)\|A^{\frac{1}{2}}Sz\|\hspace{0.05cm}\|A^{\frac{1}{2}}(SS^{\dag})^{*_A}y \| \hspace{0.3cm} (\text{by (\ref{(11)}})  )\\&\hspace{1.4cm}= \lambda(\theta_1)\|Sz\|_{A}\hspace{0.05cm}\|(SS^{\dag})^{*_A}y \|_{A}.
\end{aligned} \notag   
\end{equation}
Also, we have \begin{equation}
\begin{aligned}&
|\langle Sz,y \rangle_{A}|=|\langle S(S^{\dag}S)z,y \rangle_{A}|\\&\hspace{1.4cm}= |\langle (S^{\dag}S)z,S^{*_A}y \rangle_{A}|  \\&\hspace{1.4cm}= |\langle A^{\frac{1}{2}}(S^{\dag}S)z,A^{\frac{1}{2}}S^{*_A}y \rangle| \\&\hspace{1.4cm}\le \lambda(\theta_2)\|A^{\frac{1}{2}}(S^{\dag}S)z \|\hspace{0.05cm}  \|A^{\frac{1}{2}}S^{*_A}y\|\hspace{0.3cm}(\text{by (\ref{(11)}})  )\\&\hspace{1.4cm}= \lambda(\theta_2)\|S^{\dag}Sz \|_{A}\hspace{0.05cm}\|S^{*_A}y\|_{A}.
\end{aligned} \notag   
\end{equation}
This completes the proof.
\end{proof}

The following theorem establishes a unified upper bound for the $A-$inner product, which involves the sum of two operators $S_1$ and $S_2.$

\begin{theorem}\label{Theorem 2.3}
Let $S_1,S_2\in {\mathbb{B}_{A}(\mathcal{H})}$ with closed ranges, and let $z,y \in {\mathcal{H}}$. Then

\begin{equation} \label{(15)}
 \begin{aligned}& \mbox{\footnotesize $|\langle (S_1+S_2)z,y \rangle_{A}|\le \lambda(\theta)\sqrt{\Big\langle\Big(S_1^{*_A}S_1+(S_2^{\dag}S_2)^{*_A}S_2^{\dag}S_2\Big)z,z \Big\rangle_{A}\hspace{0.05 cm}\Big\langle\Big( S_1 S_1^{\dag}(S_1 S_1^{\dag})^{*_A}+S_2 S_2^{*_A}\Big)y,y \Big\rangle_{A}}$}\end{aligned} \tag{15}  
\end{equation}
and
\begin{equation}\label{(16)}
\begin{aligned}
&\mbox{\footnotesize$ |\langle (S_1+S_2)z,y \rangle_{A}|\le \lambda(\phi) \sqrt{\Big\langle\Big(S_2^{*_A}S_2+(S_1^{\dag}S_1)^{*_A}S_1^{\dag}S_1\Big)z,z \Big\rangle_{A}\hspace{0.05 cm}\Big\langle\Big( S_2 S_2^{\dag}(S_2 S_2^{\dag})^{*_A}+S_1 S_1^{*_A}\Big)y,y \Big\rangle}_{A},$}\end{aligned}\tag{16}\end{equation} where

\mbox{\footnotesize$\text{ $\lambda(\theta)=\max\big\{\lambda(\theta_1),\lambda(\theta_2)\big\},$ $\theta_1=\angle_{\ll A^{\frac{1}{2}}S_1 z,\hspace{0.07 cm}A^{\frac{1}{2}}(S_1 S_1^{\dag})^{*_A}y  \gg},$  $\theta_2=\angle_{\ll A^{\frac{1}{2}}S_2^{\dag}S_2z,\hspace{0.07 cm} A^{\frac{1}{2}}S_2^{*_A}y \gg}$},$} 

and

\mbox{\footnotesize$\text{$\lambda(\phi)=\max\big\{\lambda(\theta_3),\lambda(\theta_4)\big\},$ $\theta_3=\angle_{\ll A^{\frac{1}{2}}S_2 z,\hspace{0.07 cm} A^{\frac{1}{2}}(S_2 S_2^{\dag})^{*_A}y  \gg},$ $\theta_4=\angle_{\ll A^{\frac{1}{2}}S_1^{\dag}S_1z,\hspace{0.07 cm}A^{\frac{1}{2}}S_1^{*_A}y\gg}$ }.$}
\end{theorem}

\begin{proof}
Let $z,y\in{\mathcal{H}}.$
 From the inequality (\ref{(13)}), we have 
 \begin{equation}\label{(17)}
\begin{aligned}&|\langle S_1z,y \rangle_{A}|\le\lambda(\theta_1)\|S_1z\|_{A}\hspace{0.05cm}\|(S_1S_1^{\dag})^{*_A}y \|_{A} \\&\hspace{1.5cm}  \le\lambda(\theta_1)\sqrt{\langle S_1^{*_A}S_1z,z \rangle_{A}\hspace{0.07cm}\langle S_1S_1^{\dag}(S_1S_1^{\dag})^{*_A}y,y\rangle_{A}}
\end{aligned}\tag{17}\end{equation}
and from the inequality (\ref{(14)}), we get
\begin{equation}\label{(18)}
\begin{aligned}&|\langle S_2z,y \rangle_{A}|\le\lambda(\theta_2)\|S_2^{\dag}S_2z \|_{A}\hspace{0.05cm}\|S_2^{*_A}y\|_{A}\\&\hspace{1.5cm}\le\lambda(\theta_2)\sqrt{\langle (S_2^{\dag}S_2)^{*_A}S_2^{\dag}S_2z,z \rangle_{A}\hspace{0.07cm}\langle S_2S_2^{*_A}y,y \rangle_{A}}.
\end{aligned}\tag{18}\end{equation}
Now, we have 
\begin{equation*}\begin{aligned}&
 |\langle (S_1+S_2)z,y \rangle_{A}|\\&\le  |\langle S_1z,y \rangle_{A}|+ |\langle S_2z,y \rangle_{A}|\\&\mbox{\footnotesize$\le\lambda(\theta_1)\sqrt{\langle S_1^{*_A}S_1z,z \rangle_{A}\hspace{0.07cm}\langle S_1S_1^{\dag}(S_1S_1^{\dag})^{*_A}y,y\rangle_{A}}+\lambda(\theta_2)\sqrt{\langle (S_2^{\dag}S_2)^{*_A}S_2^{\dag}S_2z,z \rangle_{A}\hspace{0.07cm}\langle S_2S_2^{*_A}y,y \rangle_{A}}$}\\&\text{\hspace{0.1 cm}(by (\ref{(17)}) and (\ref{(18)}))}\\&\mbox{\footnotesize$\le\lambda(\theta)\Bigg(\sqrt{\langle S_1^{*_A}S_1 z,z \rangle_{A}\hspace{0.07 cm}\langle S_1 S_1^{\dag}(S_1 S_1^{\dag})^{*_A}y,y\rangle_{A}}+\sqrt{\langle (S_2^{\dag}S_2)^{*_A}S_2^{\dag}S_2 z,z \rangle_{A}\hspace{0.07 cm}\langle S_2 S_2^{*_A}y,y \rangle_{A}}\Bigg)$},
 \\&\text{\hspace{0.1 cm} where $\lambda(\theta)=\max\big\{\lambda(\theta_1),\lambda(\theta_2)\big\}$}.\end{aligned}\end{equation*}
 By the Cauchy-Schwarz inequality, we have
\begin{equation*}\begin{aligned} &\mbox{\footnotesize$\le\lambda(\theta) \sqrt{\Big(\langle S_1^{*_A}S_1 z,z \rangle_{A}+\langle (S_2^{\dag}S_2)^{*_A}S_2^{\dag}S_2 z,z \rangle_{A}\Big)}\sqrt{\Big(\langle S_1 S_1^{\dag}(S_1 S_1^{\dag})^{*_A}y,y\rangle_{A}+\langle S_2 S_2^{*_A} y,y \rangle_{A}\Big)}$}\\&\mbox{\footnotesize$=\lambda(\theta)\sqrt{\Big\langle\Big(S_1^{*_A}S_1+(S_2^{\dag}S_2)^{*_A}S_2^{\dag}S_2\Big)z,z \Big\rangle_{A}\hspace{0.05 cm}\Big\langle\Big( S_1 S_1^{\dag}(S_1 S_1^{\dag})^{*_A}+S_2 S_2^{*_A}\Big)y,y \Big\rangle_{A}}$}.\end{aligned}\end{equation*}
Interchanging $S_1$ and $S_2$, we infer that
\begin{equation*}\begin{aligned}
&\mbox{\footnotesize$ |\langle (S_1+S_2)z,y \rangle_{A}|\le \lambda(\phi) \sqrt{\Big\langle\Big(S_2^{*_A}S_2+(S_1^{\dag}S_1)^{*_A}S_1^{\dag}S_1\Big)z,z \Big\rangle_{A}\hspace{0.05 cm}\Big\langle\Big( S_2 S_2^{\dag}(S_2 S_2^{\dag})^{*_A}+S_1 S_1^{*_A}\Big)y,y \Big\rangle}_{A}.$}
\end{aligned}\end{equation*}
This completes the proof.
\end{proof}

From Theorem \ref{Theorem 2.3}, we directly derive the following corollaries which are considered as extended version and improved form of these inequalities ((\ref{(5)})- (\ref{(8)})).

\begin{corollary}\label{Corollary 2.4}
Let $S_1,S_2\in {\mathbb{B}_{A}(\mathcal{H})}$ with closed ranges, and let $z,y \in {\mathcal{H}}$. Then

\begin{equation} \label{(19)}
 \begin{aligned}&\| S_1+S_2 \|_{A}\le \lambda(\theta)\sqrt{\|S_1^{*_A}S_1+(S_2^{\dag}S_2)^{*_A}S_2^{\dag}S_2\|_{A}\hspace{0.09 cm}\| S_1 S_1^{\dag}(S_1 S_1^{\dag})^{*_A}+S_2 S_2^{*_A}\|_{A}}
\end{aligned} \tag{19}  
\end{equation}
and

\begin{equation}\label{(20)}
\begin{aligned}
& \|S_1+S_2\|_{A}\le \lambda(\phi) \sqrt{\|S_2^{*_A}S_2+(S_1^{\dag}S_1)^{*_A}S_1^{\dag}S_1\|_{A}\hspace{0.09 cm}\| S_2 S_2^{\dag}(S_2 S_2^{\dag})^{*_A}+S_1 S_1^{*_A}\|_{A},}\end{aligned}
\tag{20}\end{equation}where

\mbox{\footnotesize$\text{ $\lambda(\theta)=\max\big\{\lambda(\theta_1),\lambda(\theta_2)\big\},$ $\theta_1=\angle_{\ll A^{\frac{1}{2}}S_1 z,\hspace{0.07 cm}A^{\frac{1}{2}}(S_1 S_1^{\dag})^{*_A}y  \gg},$  $\theta_2=\angle_{\ll A^{\frac{1}{2}}S_2^{\dag}S_2z,\hspace{0.07 cm} A^{\frac{1}{2}}S_2^{*_A}y \gg}$},$}

and

\mbox{\footnotesize$\text{ $\lambda(\phi)=\max\big\{\lambda(\theta_3),\lambda(\theta_4)\big\},$ $\theta_3=\angle_{\ll A^{\frac{1}{2}}S_2 z,\hspace{0.07 cm} A^{\frac{1}{2}}(S_2 S_2^{\dag})^{*_A}y  \gg},$ $\theta_4=\angle_{\ll A^{\frac{1}{2}}S_1^{\dag}S_1z,\hspace{0.07 cm}A^{\frac{1}{2}}S_1^{*_A}y\gg}$ .}$} 
\end{corollary}

\begin{corollary}\label{Corollary 2.5}
Let $S_1,S_2\in {\mathbb{B}_{A}(\mathcal{H})}$ with closed ranges, and let $z \in {\mathcal{H}}$. Then

\begin{equation} \label{(21)}
 \begin{aligned}&\omega_{A}(S_1+S_2)\le \frac{\lambda(\theta)}{2}\Big\|S_1^{*_A}S_1+(S_2^{\dag}S_2)^{*_A}S_2^{\dag}S_2+S_1 S_1^{\dag}(S_1 S_1^{\dag})^{*_A}+S_2 S_2^{*_A}\Big\|_{A}\end{aligned} \tag{21}  
\end{equation}
and
\begin{equation}\label{(22)}
\begin{aligned}
& \omega_{A}(S_1+S_2)\le \frac{\lambda(\phi)}{2} \Big\|S_2^{*_A}S_2+(S_1^{\dag}S_1)^{*_A}S_1^{\dag}S_1+S_2 S_2^{\dag}(S_2 S_2^{\dag})^{*_A}+S_1 S_1^{*_A}\Big\|_{A},\end{aligned}
\tag{22}\end{equation} where

\mbox{\footnotesize$\text{ $\lambda(\theta)=\max\big\{\lambda(\theta_1),\lambda(\theta_2)\big\},$ $\theta_1=\angle_{\ll A^{\frac{1}{2}}S_1 z,\hspace{0.07 cm}A^{\frac{1}{2}}(S_1 S_1^{\dag})^{*_A}z  \gg},$ $\theta_2=\angle_{\ll A^{\frac{1}{2}}S_2^{\dag}S_2z,\hspace{0.07 cm} A^{\frac{1}{2}}S_2^{*_A}z \gg}$}$}

and

\mbox{\footnotesize$\text{ $\lambda(\phi)=\max\big\{\lambda(\theta_3),\lambda(\theta_4)\big\},$ $\theta_3=\angle_{\ll A^{\frac{1}{2}}S_2 z,\hspace{0.07 cm} A^{\frac{1}{2}}(S_2 S_2^{\dag})^{*_A}z  \gg},$ $\theta_4=\angle_{\ll A^{\frac{1}{2}}S_1^{\dag}S_1z,\hspace{0.07 cm}A^{\frac{1}{2}}S_1^{*_A}z\gg}$ }$}.

\end{corollary}
\begin{proof}
Putting $y=z$ where $\|z\|_{A}=1$ and applying AM-GM inequality in Theorem \ref{Theorem 2.3}, we reach
\begin{equation*}\begin{aligned}&
|\langle (S_1+S_2)z,z \rangle_{A}|\le \frac{\lambda(\theta)}{2}\Big\langle\Big(S_1^{*_A}S_1+(S_2^{\dag}S_2)^{*_A}S_2^{\dag}S_2+S_1 S_1^{\dag}(S_1 S_1^{\dag})^{*_A}+S_2 S_2^{*_A}\Big) z,z \Big\rangle_{A}\\&\hspace{2.5cm}\le\frac{\lambda(\theta)}{2}\Big\|S_1^{*_A}S_1+(S_2^{\dag}S_2)^{*_A}S_2^{\dag}S_2+S_1 S_1^{\dag}(S_1 S_1^{\dag})^{*_A}+S_2 S_2^{*_A}\Big\|_{A}.\end{aligned}
\end{equation*}
By considering the supremum over $z\in{\mathcal{H}}$ where $\|z\|_{A}=1,$ we get the desired inequality (\ref{(21)}).

Interchanging $S_1$ and $S_2$, we obtain the inequality (\ref{(22)}).

\end{proof}

Putting $S_1=0$ and $S_2=0$ in Corollary \ref{Corollary 2.5}, respectively we immediately obtain the next corollary.

\begin{corollary}\label{Corollary 2.6} Let $S_1\in {\mathbb{B}_{A}(\mathcal{H})}$ with closed ranges, and let $z \in {\mathcal{H}}$. Then

\begin{equation}\label{(23)}
\begin{aligned}
&\omega_{A}(S_1)\le \frac{\lambda(\theta)}{2}\min\Big\{\Big\|S_1^{*_A}S_1++S_1 S_1^{\dag}(S_1 S_1^{\dag})^{*_A}\Big\|_{A},\Big\|(S_1^{\dag}S_1)^{*_A}S_1^{\dag}S_1+S_1 S_1^{*_A}\Big\|_{A}\Big\}\\&\text{where}\\&\mbox{\footnotesize$\text{ $\lambda(\theta)=\max\big\{\lambda(\theta_1),\lambda(\theta_2)\big\},$ $\theta_1=\angle_{\ll A^{\frac{1}{2}}S_1 z,\hspace{0.07 cm}A^{\frac{1}{2}}(S_1 S_1^{\dag})^{*_A}z  \gg}$ and $\theta_2=\angle_{\ll A^{\frac{1}{2}}S_1^{\dag}S_1z,\hspace{0.07 cm} A^{\frac{1}{2}}S_1^{*_A}z \gg}$}$}.
\end{aligned}\tag{23}  
\end{equation}
\end{corollary}

\begin{remark}
It is known that $S^{\dag}S,$  $SS^{\dag}$  are self-adjoint and idempotent.
If we take $A=I$ in Corollary \ref{Corollary 2.4} and  Corollary \ref{Corollary 2.5}, respectively we notice that
\begin{equation*} 
 \begin{aligned}&\| S_1+S_2 \|\le \lambda(\theta)\sqrt{\|S_1^{*}S_1+S_2^{\dag}S_2\|\hspace{0.09 cm}\| S_1 S_1^{\dag}+S_2 S_2^{*}\|},\\&\mbox{\footnotesize$\text{ where $\lambda(\theta)=\max\big\{\lambda(\theta_1),\lambda(\theta_2)\big\},$ $\theta_1=\angle_{\ll S_1 z,\hspace{0.07 cm}S_1 S_1^{\dag}y  \gg}$ and $\theta_2=\angle_{\ll S_2^{\dag}S_2z,\hspace{0.07 cm} S_2^{*}y \gg}$}$}\\&\hspace{1.5 cm}\le \sqrt{\|S_1^{*}S_1+S_2^{\dag}S_2\|\hspace{0.09 cm}\| S_1 S_1^{\dag}+S_2 S_2^{*}\|}\text{\hspace{0.3 cm} (since $\lambda(\theta)\le1$)},
\end{aligned}  \end{equation*}

\begin{equation*}
\begin{aligned}
& \|S_1+S_2\|\le \lambda(\phi) \sqrt{\|S_2^{*}S_2+S_1^{\dag}S_1\|\hspace{0.09 cm}\| S_2 S_2^{\dag}+S_1 S_1^{*}\|,}\\&\mbox{\footnotesize$\text{ where $\lambda(\phi)=\max\big\{\lambda(\theta_3),\lambda(\theta_4)\big\},$ $\theta_3=\angle_{\ll S_2 z,\hspace{0.07 cm} S_2 S_2^{\dag}y  \gg}$ and $\theta_4=\angle_{\ll S_1^{\dag}S_1z,\hspace{0.07 cm}S_1^{*}y\gg}$ }$}\\&\hspace{1.5 cm}\le\sqrt{\|S_2^{*}S_2+S_1^{\dag}S_1\|\hspace{0.09 cm}\| S_2 S_2^{\dag}+S_1 S_1^{*}\|}\text{\hspace{0.3 cm} (since $\lambda(\phi)\le1$)}
\end{aligned}
\end{equation*}
and

\begin{equation*}
 \begin{aligned}&\omega(S_1+S_2)\le \frac{\lambda(\theta)}{2}\Big\|S_1^{*}S_1+S_2^{\dag}S_2+S_1 S_1^{\dag}+S_2 S_2^{*}\Big\|,\\&\mbox{ \footnotesize$\text{ where $\lambda(\theta)=\max\big\{\lambda(\theta_1),\lambda(\theta_2)\big\},$ $\theta_1=\angle_{\ll S_1 z,\hspace{0.07 cm} S_1 S_1^{\dag}z  \gg}$ and $\theta_2=\angle_{\ll S_2^{\dag}S_2z,\hspace{0.07 cm} S_2^{*}z \gg}$}$}\\&\hspace{1.5cm}\le\frac{1}{2}\hspace{0.05cm}\Big\|S_1^{*}S_1+S_2^{\dag}S_2+S_1 S_1^{\dag}+S_2 S_2^{*}\Big\|,
\end{aligned} \end{equation*}  
\begin{equation*}
\begin{aligned}
& \omega(S_1+S_2)\le \frac{\lambda(\phi)}{2} \Big\|S_2^{*}S_2+S_1^{\dag}S_1+S_2 S_2^{\dag}+S_1 S_1^{*}\Big\|,\\&\mbox{\footnotesize$\text{ where $\lambda(\phi)=\max\big\{\lambda(\theta_3),\lambda(\theta_4)\big\},$ $\theta_3=\angle_{\ll S_2 z,\hspace{0.07 cm} S_2 S_2^{\dag}z  \gg}$ and $\theta_4=\angle_{\ll S_1^{\dag}S_1z,\hspace{0.07 cm}S_1^{*}z\gg}$ }$}\\&\hspace{1.5cm}\le\frac{1}{2}\hspace{0.05cm}\Big\|S_2^{*}S_2+S_1^{\dag}S_1+S_2 S_2^{\dag}+S_1 S_1^{*}\Big\|.
\end{aligned}
\end{equation*}
So, we can conclude that Corollary \ref{Corollary 2.4} and  Corollary \ref{Corollary 2.5} extend and improve the inequalities ((\ref{(5)})-(\ref{(8)})).
\end{remark}

\begin{lemma}\label{Lemma 2.8}

[\citenum{conde2024some}, Lemma 3.3]
Let $T_1\in{\mathbb{B}(\mathcal{H})}$ be an $A-$positive operator and $y\in{\mathcal{H}}$ with $\|y\|_{A}=1.$ Then \begin{equation*}\langle T_1 y,y \rangle _{A}^n\le \langle T_1^n y,y \rangle _{A} \hspace{0.4cm}\forall n \in {\mathbb{N}}.\end{equation*}
\end{lemma}

\begin{theorem} \label{Theorem 2.9}
Let $S_1,S_2\in {\mathbb{B}_{A}(\mathcal{H})}$ with closed ranges, and let $z \in {\mathcal{H}}$. Then
\begin{equation}\label{(24)}
\begin{aligned}\omega_{A}^2(S_1+S_2)\le &\mbox{\footnotesize$2\lambda(\theta)
\min\Bigg\{\sqrt{\left\|\left(S_1^{*_A}S_1\right)^2+\left(S_2S_2^{\dag}(S_2S_2^{\dag})^{*_A}\right)^2\right\|_{A}\hspace{0.04cm}\left\|\left(S_2^{*_A}S_2\right)^2+\left(S_1S_1^{\dag}(S_1S_1^{\dag})^{*_A}\right)^2\right\|_{A}},$}\\&\mbox{\footnotesize$\sqrt{\left\|\left(S_1^{*_A}S_1\right)^2+\left(S_2^{*_A}S_2\right)^2\right\|_{A}\hspace{0.04cm}\left\|\left(S_1S_1^{\dag}(S_1S_1^{\dag})^{*_A}\right)^2+\left(S_2S_2^{\dag}(S_2S_2^{\dag})^{*_A}\right)^2\right\|_{A}}\Bigg\}$},
\end{aligned}\tag{24}\end{equation}
 
and

\begin{equation}\label{(25)}
\begin{aligned}\omega_{A}^2(S_1+S_2)\le &\mbox{\footnotesize$2\lambda(\phi)
\min\Bigg\{\sqrt{\left\|\left(S_1S_1^{*_A}\right)^2+\left((S_2^{\dag}S_2)^{*_A}S_2^{\dag}S_2\right)^2\right\|_{A}\left\|\left(S_2S_2^{*_A}\right)^2+\left((S_1^{\dag}S_1)^{*_A}S_1^{\dag}S_1\right)^2\right\|_{A}},$}\\&\mbox{\footnotesize$\sqrt{\left\|\left(S_1S_1^{*_A}\right)^2+\left(S_2S_2^{*_A}\right)^2\right\|_{A}\left\|\left((S_1^{\dag}S_1)^{*_A}S_1^{\dag}S_1\right)^2+\left((S_2^{\dag}S_2)^{*_A}S_2^{\dag}S_2\right)^2\right\|_{A}}\Bigg\}$},
\end{aligned}\tag{25}\end{equation}
where

$\lambda(\theta)=\max\{\lambda^2(\theta_1), \lambda^2(\theta_2)\},$ $\lambda(\phi)=\max\{\lambda^2(\theta_3), \lambda^2(\theta_4)\},$

$\theta_1=\angle_{\ll A^{\frac{1}{2}}S_1 z,\hspace{0.07cm}A^{\frac{1}{2}}(S_1 S_1^{\dag})^{*_A}z  \gg},$  $\theta_2=\angle_{\ll A^{\frac{1}{2}}S_2 z,\hspace{0.07cm}A^{\frac{1}{2}}(S_2 S_2^{\dag})^{*_A}z \gg}$ 

$\theta_3=\angle_{\ll A^{\frac{1}{2}}S_1^{*_A} z,\hspace{0.07cm}A^{\frac{1}{2}} S_1^{\dag}S_1z  \gg},$  $\theta_4=\angle_{\ll A^{\frac{1}{2}}S_2^{*_A} z,\hspace{0.07cm}A^{\frac{1}{2}}S_2^{\dag}S_2z \gg}$.
\end{theorem}

\begin{proof}
Let $z\in{\mathcal{H}}$ where $\|z\|_{A}=1.$ Then we have
\begin{equation}\begin{aligned}
   &|\langle (S_1+S_2)z,z \rangle_{A}|^2\\&\le ( |\langle S_1z,z \rangle_{A}|+ |\langle S_2z,z\rangle_{A}|)^2\\&\le 2( |\langle S_1z,z \rangle_{A}|^2+ |\langle S_2z,z\rangle_{A}|^2)\\&\mbox{\footnotesize$\le2\left(\lambda^2(\theta_1)\langle S_1^{*_A}S_1z,z \rangle_{A}\hspace{0.07cm}\langle S_1S_1^{\dag}(S_1S_1^{\dag})^{*_A}z,z\rangle_{A}+\lambda^2(\theta_2)\langle S_2^{*_A}S_2z,z \rangle_{A}\hspace{0.07cm}\langle S_2S_2^{\dag}(S_2S_2^{\dag})^{*_A}z,z\rangle_{A}\right)$}\\&\mbox{\footnotesize$\ \le 2\lambda(\theta)\left(\langle S_1^{*_A}S_1z,z \rangle_{A}\hspace{0.07cm}\langle S_1S_1^{\dag}(S_1S_1^{\dag})^{*_A}z,z\rangle_{A}+\langle S_2^{*_A}S_2z,z \rangle_{A}\hspace{0.07cm}\langle S_2S_2^{\dag}(S_2S_2^{\dag})^{*_A}z,z\rangle_{A}\right)$}\end{aligned}\notag\end{equation}
 \begin{equation}\begin{aligned}&\mbox{\footnotesize$\le 2\lambda(\theta) \sqrt{\Big(\langle S_1^{*_A}S_1 z,z \rangle_{A}^2+\langle S_2S_2^{\dag}(S_2S_2^{\dag})^{*_A} z,z \rangle_{A}^2\Big)}\sqrt{\Big(\langle S_1 S_1^{\dag}(S_1 S_1^{\dag})^{*_A}z,z\rangle_{A}^2+\langle S_2^{*_A} S_2 z,z \rangle_{A}^2\Big)}$}\\&\mbox{\footnotesize$\le 2\lambda(\theta)\sqrt{\Big\langle\Big(\left(S_1^{*_A}S_1\right)^2+\left(S_2S_2^{\dag}(S_2S_2^{\dag})^{*_A}\right)^2\Big)z,z \Big\rangle_{A}\hspace{0.05 cm}\Big\langle\Big( \left(S_1 S_1^{\dag}(S_1 S_1^{\dag})^{*_A}\right)^2+\left( S_2^{*_A}S_2\right)^2\Big)z,z\Big\rangle_{A}}.$} \end{aligned}\notag\end{equation}
 So, \begin{equation}\begin{aligned}\label{(26)}
  \mbox{\footnotesize$|\langle (S_1+S_2)z,z \rangle_{A}|^2\le 2 \lambda(\theta)\left\|\left(S_1^{*_A}S_1\right)^2+\left(S_2S_2^{\dag}(S_2S_2^{\dag})^{*_A}\right)^2\right\|_{A}^{\frac{1}{2}}\hspace{0.04cm}\left\|\left(S_2^{*_A}S_2\right)^2+\left(S_1S_1^{\dag}(S_1S_1^{\dag})^{*_A}\right)^2\right\|_{A}^{\frac{1}{2}}$},
  \end{aligned}\tag{26}\end{equation}
where the first inequality follows from the triangle inequality, the second inequality follows from the convexity of the function $h(t)=t^2,$ the third inequality is obtained from the inequality (\ref{(13)}) with 
 $\theta_1=\angle_{\ll A^{\frac{1}{2}}S_1 z,\hspace{0.07cm}A^{\frac{1}{2}}(S_1 S_1^{\dag})^{*_A}z  \gg}$ and $\theta_2=\angle_{\ll A^{\frac{1}{2}}S_2 z,\hspace{0.07cm}A^{\frac{1}{2}}(S_2 S_2^{\dag})^{*_A}z   \gg},$ in the  fourth inequality we consider $\lambda(\theta)=\max\big\{\lambda^2(\theta_1),\lambda^2(\theta_2)\big\},$
we use the fact $d_1d_2+d_3d_4\le\sqrt{(d_1^2+d_3^2)(d_2^2+d_4^2)}$ for $d_1,d_2,d_3,d_4\in{\mathbb{R}}$ to obtain the fifth inequality and we apply Lemma \ref{Lemma 2.8} to obtain the sixth inequality.

 In the above proof, by taking different $d_1,d_2,d_3,d_4$ we reach at:
 \begin{equation}\label{(27)}
 \mbox{\footnotesize$|\langle (S_1+S_2)z,z \rangle_{A}|^2\le   2 \lambda(\theta)\left\|\left(S_1^{*_A}S_1\right)^2+\left(S_2^{*_A}S_2\right)^2\right\|_{A}^{\frac{1}{2}}\hspace{0.04cm}\left\|\left(S_1S_1^{\dag}(S_1S_1^{\dag})^{*_A}\right)^2+\left(S_2S_2^{\dag}(S_2S_2^{\dag})^{*_A}\right)^2\right\|_{A}^{\frac{1}{2}}$}  
\tag{27}\end{equation}

By considering the supremum over $z\in{\mathcal{H}}$ where $\|z\|_{A}=1$ in (\ref{(26)}) and (\ref{(27)}), we obtain the 
inequality (\ref{(24)}).

Again by applying the inequality (\ref{(14)}) and proceeding similarly as the above proof, we can get the desired inequality (\ref{(25)}).
\end{proof}

\begin{remark}
Letting $S_1=S_2$  in Theorem \ref{Theorem 2.9}, we establish that
\begin{equation*}
\begin{aligned}\omega^2_{A}(S_1)&\le \mbox{\footnotesize$\frac{\lambda(\theta)}{2}
\min\Bigg\{\left\|\left(S_1^{*_A}S_1\right)^2+\left(S_1S_1^{\dag}(S_1S_1^{\dag})^{*_A}\right)^2\right\|_{A},2\left\|S_1^{*_A}S_1\right\|_{A}\hspace{0.04cm}\left\|S_1S_1^{\dag}(S_1S_1^{\dag})^{*_A}\right\|_{A}\Bigg\}$}\\&\le \lambda(\theta)\|S_1\|_{A}^2, \text{which implies \hspace{0.1cm} $\omega_{A}(S_1)\le \|S_1\|_{A}$}
\end{aligned}\end{equation*}
 
and \begin{equation*}
\begin{aligned}\omega^2_{A}(S_1)&\le \mbox{\footnotesize$\frac{\lambda(\phi)}{2}\min\Bigg\{\left\|\left(S_1S_1^{*_A}\right)^2+\left((S_1^{\dag}S_1)^{*_A}S_1^{\dag}S_1\right)^2\right\|_{A},2\left\|S_1S_1^{*_A}\right\|_{A}\hspace{0.03cm}\left\|(S_1^{\dag}S_1)^{*_A}S_1^{\dag}S_1\right\|_{A}\Bigg\}$}\\&\le \lambda(\phi)\|S_1\|_{A}^2, \text{which implies \hspace{0.1cm} $\omega_{A}(S_1)\le \|S_1\|_{A}$.}
\end{aligned}\end{equation*}
\end{remark}

\begin{theorem}\label{Theorem 2.11}
 Let $S_1,S_2\in {\mathbb{B}_{A}(\mathcal{H})}$ with closed ranges, and let $z \in {\mathcal{H}}$. Then
\begin{equation}\label{(28)}
\begin{aligned}\omega_{A}^2(S_1+S_2)\le &\mbox{\footnotesize$\lambda(\theta)
\min\Bigg\{\left\|S_1^{*_A}S_1+(S_2^{\dag}S_2)^{*_A}S_2^{\dag}S_2\right\|_{A}\left\|S_2S_2^{*_A}+S_1S_1^{\dag}(S_1S_1^{\dag})^{*_A}\right\|_{A},$}\\&\mbox{\footnotesize$\left\|S_1^{*_A}S_1+S_2S_2^{*_A}\right\|_{A}\left\|S_1S_1^{\dag}(S_1S_1^{\dag})^{*_A}+(S_2^{\dag}S_2)^{*_A}S_2^{\dag}S_2\right\|_{A}$}\Bigg\},
\end{aligned}\tag{28}\end{equation}
where   $\lambda(\theta)=\max\{\lambda^2(\theta_1), \lambda^2(\theta_2)\},$

$\theta_1=\angle_{\ll A^{\frac{1}{2}}S_1 z,\hspace{0.07cm}A^{\frac{1}{2}}(S_1 S_1^{\dag})^{*_A}z  \gg}$ and $\theta_2=\angle_{\ll A^{\frac{1}{2}}S_2^{*_A} z,\hspace{0.07cm}A^{\frac{1}{2}}S_2^{\dag}S_2z \gg}$.
   \end{theorem}

\begin{proof}
Let $z\in{\mathcal{H}}$ where $\|z\|_{A}=1.$
We have \begin{equation*}
  |\langle (S_1+S_2)z,z \rangle_{A}|^2\le ( |\langle S_1z,z \rangle_{A}|+ |\langle S_2z,z\rangle_{A}|)^2.  
\end{equation*}
By using inequalities (\ref{(13)}) and (\ref{(14)}), respectively we arrive at  
\begin{equation*}
|\langle S_1z,z \rangle_{A}|\le  \lambda(\theta_1)\sqrt{\langle S_1^{*_A}S_1z,z \rangle_{A}\hspace{0.07cm}\langle S_1S_1^{\dag}(S_1S_1^{\dag})^{*_A}z,z\rangle_{A}},   
\end{equation*}
where $\theta_1=\angle_{\ll A^{\frac{1}{2}}S_1 z,\hspace{0.07cm}A^{\frac{1}{2}}(S_1 S_1^{\dag})^{*_A}z  \gg}$ 

and

\begin{equation*}
|\langle S_2z,z \rangle_{A}|\le  \lambda(\theta_2)\sqrt{\langle S_2S_2^{*_A}z,z \rangle_{A}\hspace{0.07cm}\langle (S_2^{\dag}S_2)^{*_A}S_2^{\dag}S_2z,z\rangle_{A}},   
\end{equation*}
where $\theta_2=\angle_{\ll A^{\frac{1}{2}}S_2^{*_A} z,\hspace{0.07cm}A^{\frac{1}{2}}S_2^{\dag}S_2z \gg}$. 

Utilizing the fact that $d_1d_2+d_3d_4\le\sqrt{(d_1^2+d_3^2)(d_2^2+d_4^2)}$ forall $d_1,d_2,d_3,d_4\in{\mathbb{R}},$ we get that
\begin{equation}
\begin{aligned}& \Big( |\langle S_1z,z \rangle_{A}|+ |\langle S_2z,z\rangle_{A}|\Big)^2
\\&\le\lambda(\theta) \left\langle \left(S_1^{*_A}S_1+S_2S_2^{*_A}\right)z,z\right\rangle_{A}\hspace{0.04cm}\left\langle \left( S_1S_1^{\dag}(S_1S_1^{\dag})^{*_A}+ (S_2^{\dag}S_2)^{*_A}S_2^{\dag}S_2\right)z,z\right\rangle_{A}\\&\le\lambda(\theta)\left\|S_1^{*_A}S_1+S_2S_2^{*_A}\right\|_{A}\left\|S_1S_1^{\dag}(S_1S_1^{\dag})^{*_A}+(S_2^{\dag}S_2)^{*_A}S_2^{\dag}S_2\right\|_{A}.\end{aligned}\notag\end{equation}
Considering different $d_1,d_2,d_3,d_4$ we also obtain that
\begin{equation}
\begin{aligned}& \Big( |\langle S_1z,z \rangle_{A}|+ |\langle S_2z,z\rangle_{A}|\Big)^2
\\&\le \lambda(\theta)\left\|S_1^{*_A}S_1+(S_2^{\dag}S_2)^{*_A}S_2^{\dag}S_2\right\|_{A}\left\|S_2S_2^{*_A}+S_1S_1^{\dag}(S_1S_1^{\dag})^{*_A}\right\|_{A}
\end{aligned}\notag\end{equation}

where   $\lambda(\theta)=\max\{\lambda^2(\theta_1), \lambda^2(\theta_2)\}.$

By deciding on the supremum over $z\in{\mathcal{H}}$ where $\|z\|_{A}=1$ in the above two inequalities, we obtain the inequality (\ref{(28)}).
\end{proof}

\begin{remark}
Letting   $S_1=S_2$  in Theorem \ref{Theorem 2.11}, we infer that
\begin{equation}\begin{aligned}
\omega^2_{A}(S_1)&\le\frac{\lambda(\theta)}{4}  \mbox{\footnotesize$
\min\Bigg\{\left\|S_1^{*_A}S_1+(S_1^{\dag}S_1)^{*_A}S_1^{\dag}S_1\right\|_{A}\left\|S_1S_1^{*_A}+S_1S_1^{\dag}(S_1S_1^{\dag})^{*_A}\right\|_{A},$}\\&\mbox{\footnotesize$\left\|S_1^{*_A}S_1+S_1S_1^{*_A}\right\|_{A}\left\|S_1S_1^{\dag}(S_1S_1^{\dag})^{*_A}+(S_1^{\dag}S_1)^{*_A}S_1^{\dag}S_1\right\|_{A}\Bigg\}$}\\&\le\frac{\lambda(\theta)}{4}  \mbox{\footnotesize$
\min\Bigg\{\left\|S_1^{*_A}S_1+(S_1^{\dag}S_1)^{*_A}S_1^{\dag}S_1\right\|_{A}\left\|S_1S_1^{*_A}+S_1S_1^{\dag}(S_1S_1^{\dag})^{*_A}\right\|_{A},$}\\&\mbox{\footnotesize$2\left\|S_1^{*_A}S_1+S_1S_1^{*_A}\right\|_{A}\Bigg\}$}\\&\le\frac{1}{2}\left\|S_1^{*_A}S_1+S_1S_1^{*_A}\right\|_{A}.\end{aligned}\notag\end{equation}
\end{remark}

\section{ \texorpdfstring{$\mathbb{A-}$D} Davis-Wielandt radius inequalities for operator matrices and \texorpdfstring{$A-$w} weighted numerical radius bounds}\label{sec3}
We take the diagonal operator matrix working on the Hilbert space $\mathbb{H}=\sum\limits_{i=1}^n\bigoplus\mathcal{H}$ ( direct sum of $n-$copies of $\mathcal{H}$), where each diagonal entry is the positive operator $A:$

$\mathbb{A}= \begin{bmatrix}
    
A&&\\
        &\ddots& \\
          &&A
\end{bmatrix}.$
The positive operator $\mathbb{A}$ induces the following semi-inner product:
$$\langle z,w \rangle_{\mathbb{A}}=\langle \mathbb{A}z,w \rangle=\sum_{j=1}^n \langle Az_{j},w_{j} \rangle=\sum_{j=1}^n \langle z_{j},w_{j} \rangle_{A},$$ for all
$z=(z_1,z_2,\cdots,z_n)$ and $w=(w_1,w_2,\cdots,w_n)\in{\mathbb{H}}.$

Various numerical radius inequalities for $n\times n$ operator matrices have been established by different authors, as referenced in [\citenum{abu2015numerical, bani2009numerical, hou1995norm, bhunia2024sharper}].

In 2023, some Davis-Wielandt radius inequalities for $n\times n$ operator matrices are derived by M.W. Alomari, which is demonstrated in [\citenum{alomari2023davis}].

In our work, we generalize and refine the existing Davis-Wielandt radius inequalities estimated by M.W. Alomari in [\citenum{alomari2023davis}].

To establish our new results we need the following lemmas.
\begin{lemma}\label{Lemma 3.1}

[\citenum{bohr1914theorem}]
For $j=1,2,\cdots,k,$ let $b_j$   be a positive real numbers. Then \begin{equation*}
 \left(\sum\limits_{j=1}^k b_j\right)^n \le k^{n-1}\sum\limits_{j=1}^k b_{j}^n\hspace{1cm} \text{ for all $n\ge1.$}   
\end{equation*}
\end{lemma}

\begin{lemma}\label{Lemma 3.2}

[\citenum{qiao2022numerical}, Lemma 2.4]
Let $S\in {\mathbb{B}_{A}(\mathcal{H})}.$  Then 
\begin{equation*}
 |\langle Sy,z \rangle_{A}|^2\le \|Sy\|_{A}\hspace{0.05cm}\|S^{*_A}z\|_{A}=\sqrt{\langle S^{*_A}Sy,y\rangle_{A}}\sqrt{\langle SS^{*_A}z,z\rangle_{A}}   
\end{equation*} for any $y,z\in{\mathcal{H}}$  where $\|y\|_{A}=\|z\|_{A}=1.$
\end{lemma}

\begin{lemma}\label{Lemma 3.3}

[\citenum{horn1991topics}, p.44]
 Let $S=[s_{ij}]\in{\mathbb{M}_{n}(\mathbb{C})}$  such that $s_{ij}\ge0$ for  all $i,j=1,2,\cdots,n.$ Then \begin{equation*}
 \omega(S)=\frac{1}{2}r(s_{ij}+s_{ji}).    \end{equation*}
\end{lemma}

In the following result, we establish a new $\mathbb{A}-$ Davis-Wielandt radius inequality for operator matrices.

\begin{theorem}\label{Theorem 3.4}
Let $S_{ij}\in{\mathbb{B}_{A}(\mathcal{H})}$ for all $i,j=1,2,\cdots,n.$ Then
\begin{equation}
\begin{aligned}\label{(29)}
d\omega_{\mathbb{A}}\left(\begin{bmatrix}
S_{11} & S_{12} & \cdots & S_{1n} \\
S_{21} & S_{22} & \cdots & S_{2n} \\
\vdots & \vdots & \ddots & \vdots \\
S_{n1} & S_{n2} & \cdots & S_{nn}
\end{bmatrix}\right)\le \sqrt{\omega(\mathbb{\breve{S}})}    
\end{aligned}  \tag{29}  
\end{equation}
 where $\mathbb{\breve{S}}=[\widetilde{s_{ij}}]$ and
 \begin{equation}
    \widetilde{s_{ij}}=n\cdot\begin{cases}
\omega_{A}^2\left(S_{ii}\right)+\|S_{ii}\|_{A}^4  & \text{if } i=j \\ \mbox{\footnotesize$\sqrt{\left\|S^{*_A}_{ij}S_{ij}+S_{ji}S^{*_A}_{ji}\right\|_{A}\left\|S_{ij}S^{*_A}_{ij}+S^{*_A}_{ji}S_{ji}\right\|_{A}}+\left\|(S^{*_A}_{ij}S_{ij})^2+(S^{*_A}_{ji}S_{ji})^2\right\|_{A}$}
& \text{if } i<j\\0 &  \text{elsewhere}\end{cases}.\notag\end{equation}

\end{theorem}

\begin{proof}
Let $z =[z_1, z_2,  \cdots  ,z_n
]^{t}\in{\mathbb{H}}$ 
where $\|z\|_{\mathbb{A}}=1$
and
$\tilde{z}_{A}=\begin{bmatrix}
\|z_1\|_{A}, & \|z_2\|_{A}, & \cdots & ,\|z_n\|_{A}
\end{bmatrix}^{t}.$ 

Then it is clear that $\tilde{z}_{A}\in{\mathbb{C}^n}$ is a unit vector.
Then we have,
\begin{equation}\begin{aligned}&|\langle Sz,z \rangle_{\mathbb{A}} |^2\\&= \left|\sum\limits_{i,j=1}^{n} \langle S_{ij}z_j,z_i \rangle_{A}\right|^2\\&\le n\cdot\sum\limits_{i,j=1}^{n} \left|\langle S_{ij}z_j,z_i \rangle_{A}\right|^2\hspace{1cm}(\text{by Lemma \ref{Lemma 3.1}})\\&=n \cdot\sum\limits_{i=1}^{n} \left|\langle S_{ii}z_i,z_i \rangle_{A}\right|^2+n \cdot\sum\limits_{\substack{i,j=1\\i \ne j}}^{n} \left|\langle S_{ij}z_j,z_i \rangle_{A}\right|^2\\&=n\cdot\sum\limits_{i=1}^{n} \left|\langle S_{ii}z_i,z_i \rangle_{A}\right|^2+n\cdot\sum\limits_{\substack{i,j=1\\i < j}}^{n}\big(\hspace{0.05cm}\left|\langle S_{ij}z_j,z_i \rangle_{A}\right|^2+\left|\langle S_{ji}z_i,z_j \rangle_{A}\right|^2\hspace{0.05cm}\big).\end{aligned}\notag\end{equation}
By using  Lemma \ref{Lemma 3.2}, we see 
\begin{equation}\begin{aligned}
|\langle Sz,z \rangle_{\mathbb{A}} |^2&\le n\cdot\sum\limits_{i=1}^{n} \left|\langle S_{ii}z_i,z_i \rangle_{A}\right|^2+n\cdot\sum\limits_{\substack{i,j=1\\i < j}}^{n}\Bigg(\hspace{0.05cm}\sqrt{\left\langle S^{*_A}_{ij}S_{ij}z_j,z_j\right\rangle_{A}\left\langle S_{ij}S^{*_A}_{ij}z_i,z_i\right\rangle_{A}}+\\&\sqrt{\left\langle S^{*_A}_{ji}S_{ji}z_i,z_i\right\rangle_{A}\left\langle S_{ji}S^{*_A}_{ji}z_j,z_j\right\rangle_{A}}\hspace{0.05cm}\Bigg).\end{aligned}\notag\end{equation}
By using the fact $d_1d_2+d_3d_4\le\sqrt{(d_1^2+d_3^2)(d_2^2+d_4^2)}$, for $d_1,d_2,d_3,d_4\ge0$, we notice that
\begin{equation}\label{(30)}\begin{aligned}
|\langle Sz,z \rangle_{\mathbb{A}} |^2&\le n\cdot\sum\limits_{i=1}^{n} \left|\langle S_{ii}z_i,z_i \rangle_{A}\right|^2+n\cdot\sum\limits_{\substack{i,j=1\\i < j}}^{n}\Bigg(\hspace{0.05cm}\sqrt{\left\langle\left( S^{*_A}_{ij}S_{ij}+S_{ji}S^{*_A}_{ji}\right)z_j,z_j\right\rangle_{A}}\times\\&\sqrt{\left\langle \left(S_{ij}S^{*_A}_{ij}+S^{*_A}_{ji}S_{ji}\right)z_i,z_i\right\rangle_{A}}\Bigg)\hspace{0.01cm}\\&\mbox{\footnotesize $\le n \cdot\sum\limits_{i=1}^{n}\omega^2_{A}(S_{ii})\|z_{i}\|_{A}^2+n\cdot\sum\limits_{\substack{i,j=1\\i < j}}^{n}\sqrt{\left\|S^{*_A}_{ij}S_{ij}+S_{ji}S^{*_A}_{ji}\right\|_{A}\left\|S_{ij}S^{*_A}_{ij}+S^{*_A}_{ji}S_{ji}\right\|_{A}}\|z_{i}\|_{A}\|z_{j}\|_{A}$}.
\end{aligned}\tag{30}\end{equation}
By the similar approach, we derive
\begin{equation}\label{(31)}\begin{aligned}&|\langle S^{*_A}Sz,z \rangle_{\mathbb{A}} |^2\\&= \left|\sum\limits_{i,j=1}^{n} \langle S^{*_A}_{ij}S_{ij}z_j,z _i\rangle_{A}\right|^2\\&\le n \cdot\sum\limits_{i=1}^{n}\omega^2_{A}(S^{*_A}_{ii}S_{ii})\|z_{i}\|_{A}^2+n\cdot\sum\limits_{\substack{i,j=1\\i < j}}^{n}\Bigg(\sqrt{\left\|(S^{*_A}_{ij}S_{ij})^{*_A}S^{*_A}_{ij}S_{ij}+S^{*_A}_{ji}S_{ji}(S^{*_A}_{ji}S_{ji})^{*_A}\right\|_{A}}\times\\&\sqrt{\left\|S^{*_A}_{ij}S_{ij}(S^{*_A}_{ij}S_{ij})^{*_A}+(S^{*_A}_{ji}S_{ji})^{*_A}S^{*_A}_{ji}S_{ji}\right\|_{A}}\Bigg)\|z_{i}\|_{A}\|z_{j}\|_{A}.
\end{aligned}\tag{31}\end{equation}

Since $S^{*_A}_{ii}S_{ii},$ $S^{*_A}_{ij}S_{ij}$ and $S^{*_A}_{ji}S_{ji}$ are $A-$ self-adjoint, we can write the inequality (\ref{(31)}) as
\begin{equation}\label{(32)}
 \begin{aligned}
  |\langle S^{*_A}Sz,z \rangle_{\mathbb{A}} |^2\le n \cdot\sum\limits_{i=1}^{n}\|S_{ii}\|_{A}^4+n\cdot\sum\limits_{\substack{i,j=1\\i < j}}^{n}\left\|(S^{*_A}_{ij}S_{ij})^2+(S^{*_A}_{ji}S_{ji})^2\right\|_{A}\|z_{i}\|_{A}\|z_{j}\|_{A}.   
 \end{aligned}   
\tag{32}\end{equation}

We notice that, \begin{equation}\begin{aligned}
d\omega_{\mathbb{A}}(S)&=\sup\left\{\sqrt{|\langle Sz,z \rangle_{\mathbb{A}}|^2+\|Sz\|_{\mathbb{A}}^4}: z\in{\mathbb{H}}, \|z\|_{\mathbb{A}}=1\right\}\\& =\sup\left\{\sqrt{|\langle Sz,z \rangle_{\mathbb{A}}|^2+|\langle S^{*_\mathbb{A}}Sz,z \rangle_{\mathbb{A}}|^2}: z\in{\mathbb{H}}, \|z\|_{\mathbb{A}}=1\right\}.\end{aligned}\notag\end{equation}
Adding (\ref{(30)}), (\ref{(32)}) and then considering supremum over $z\in{\mathbb{H}}$ with $\|z\|_{\mathbb{A}}=1,$ we get that
\begin{equation}\begin{aligned}
d\omega_{\mathbb{A}}^2(S)&\le\sum\limits_{i,j=1}^n\widetilde{s_{ij}}\|z_i\|_{A}\hspace{0.06cm}\|z_j\|_{A}\\&=\langle\mathbb{\breve{S}} \tilde{z}_{A},\tilde{z}_{A}\rangle\\&=\omega(\mathbb{\breve{S}}) .    
\end{aligned}\notag\end{equation}
\end{proof}

Based on Theorem \ref{Theorem 3.4} and utilizing Lemma \ref{Lemma 3.3}, we find the following bound.

\begin{corollary}\label{Corollary 3.5}
Let $S_{ij}\in{\mathbb{B}_{A}(\mathcal{H})}$ for all $i,j=1,2.$
\begin{equation}\label{(33)}
d\omega_{\mathbb{A}}\left(\begin{bmatrix}
S_{11} & S_{12}  \\
S_{21} & S_{22} 
\end{bmatrix}\right)\le \sqrt{a_1+a_3+\sqrt{(a_1-a_3)^2+a_2^{2}}}, 
\tag{33}\end{equation}
where 
\begin{equation*}\begin{aligned}
&a_1=\omega_{A}^2\left(S_{11}\right)+\|S_{11}\|_{A}^4, \hspace{1cm} a_3=\omega_{A}^2\left(S_{22}\right)+\|S_{22}\|_{A}^4,\\&a_2=\mbox{\footnotesize$\sqrt{\left\|S^{*_A}_{12}S_{12}+S_{21}S^{*_A}_{21}\right\|_{A}\left\|S_{12}S^{*_A}_{12}+S^{*_A}_{21}S_{21}\right\|_{A}}+\left\|(S^{*_A}_{12}S_{12})^2+(S^{*_A}_{21}S_{21})^2\right\|_{A}$}.\end{aligned}\end{equation*}
\end{corollary}

\begin{proof}
Letting  $a_1,a_2,a_3$ as described above and $n=2$ in   Theorem \ref{Theorem 3.4}, we get that
\begin{equation}\begin{aligned}d\omega_{\mathbb{A}}\left(\begin{bmatrix}
S_{11} & S_{12}  \\
S_{21} & S_{22} 
\end{bmatrix}\right)&\le2\omega\left(\begin{bmatrix}
a_1 &  a_2 \\
 0& a_3 
\end{bmatrix}\right)\\&=2r\left(\begin{bmatrix}
a_1 & \frac{a_2}{2} \\
 \frac{a_2}{2} & a_3 
\end{bmatrix}\right)\hspace{1cm}(\text{Lemma \ref{Lemma 3.3}})\\&=a_1+a_3+\sqrt{(a_1-a_3)^2+a_2^{2}},\end{aligned}\notag\end{equation} which provides the desired inequality (\ref{(33)}).
\end{proof}

\begin{remark}
Considering $A=I$ in Corollary  \ref{Corollary 3.5}, the inequality (\ref{(33)}) becomes \begin{equation}
d\omega\left(\begin{bmatrix}
S_{11} & S_{12}  \\
S_{21} & S_{22} 
\end{bmatrix}\right)\le \sqrt{a_1+a_3+\sqrt{(a_1-a_3)^2+a_2^{2}}}, 
\notag\end{equation}
where 
\begin{equation*}\begin{aligned}
&a_1=\omega^2\left(S_{11}\right)+\|S_{11}\|^4, \hspace{1cm} a_3=\omega^2\left(S_{22}\right)+\|S_{22}\|^4,\\&a_2=\mbox{\footnotesize$\sqrt{\left\||S_{12}|^2+|S^{*}_{21}|^2\right\|\left\||S^{*}_{12}|^2+|S_{21}|^2\right\|}+\left\||S_{12}|^4+|S_{21}|^4\right\|$}.\end{aligned}\end{equation*}
Now, \begin{equation*}
 \begin{aligned}
a_2&=\mbox{\footnotesize$\sqrt{\left\||S_{12}|^2+|S^{*}_{21}|^2\right\|\left\||S^{*}_{12}|^2+|S_{21}|^2\right\|}+\left\||S_{12}|^4+|S_{21}|^4\right\|$} \\&\le \|S_{12}\|^2+\|S_{21}\|^2+\|S_{12}\|^4+\|S_{21}\|^4\\&=a_4 \hspace{0.02cm}(\text{say}).
 \end{aligned}\end{equation*}

We address a numerical example:

Let $S_{11}=\begin{bmatrix}
    1&1\\0&1
\end{bmatrix},$ $S_{12}=\begin{bmatrix}
    1&2\\1&0
\end{bmatrix},$ $S_{21}=\begin{bmatrix}
    2&1\\1&1
\end{bmatrix}$ and $S_{22}=\begin{bmatrix}
    1&0\\2&1
\end{bmatrix}.$

By the calculation we obtain:
$a_1=9.1041,$ $ a_2=82.4745,$  $a_3=37.9706$ and $a_4=86.4853.$
 For this example, we get 
 \begin{align*}
d\omega\left(\begin{bmatrix}
S_{11} & S_{12}  \\
S_{21} & S_{22} 
\end{bmatrix}\right)\le \sqrt{a_1+a_3+\sqrt{(a_1-a_3)^2+a_2^{2}}}= 11.5955
\end{align*} and from
[\citenum{alomari2023davis}, Corollary 3.4], we estimate
\begin{align*}
d\omega\left(\begin{bmatrix}
S_{11} & S_{12}  \\
S_{21} & S_{22} 
\end{bmatrix}\right)\le \sqrt{a_1+a_3+\sqrt{(a_1-a_3)^2+a_4^{2}}}= 11.758.
\end{align*}
 Thus, the inequality (\ref{(33)}) generalizes and improves the inequality which is given in [\citenum{alomari2023davis}, Corollary 3.4].

\end{remark}

In the following result, we provide another $\mathbb{A}-$ Davis-Wielandt radius inequality for operator matrices.

\begin{theorem}\label{Theorem 3.7}
Let $S_{ij}\in{\mathbb{B}_{A}(\mathcal{H})}$ for all $i,j=1,2,\cdots,n.$ Then
\begin{equation}
\begin{aligned}\label{(34)}
d\omega_{\mathbb{A}}\left(\begin{bmatrix}
S_{11} & S_{12} & \cdots & S_{1n} \\
S_{21} & S_{22} & \cdots & S_{2n} \\
\vdots & \vdots & \ddots & \vdots \\
S_{n1} & S_{n2} & \cdots & S_{nn}
\end{bmatrix}\right)\le \omega(\mathbb{\breve{S}})    
\end{aligned}  \tag{34}  
\end{equation}
 where $\mathbb{\breve{S}}=[\widetilde{s_{ij}}]$ and
 \begin{equation}
    \widetilde{s_{ij}}=\begin{cases}
\omega_{A}\left(S_{ii}\right)+\|S_{ii}\|_{A}^2 & \text{if } i=j \\ \|S_{ij} \|_{A}+\|S_{ji}\|_{A}+\|S_{ij} \|^2_{A}+\|S_{ji}\|^2_{A}
& \text{if } i<j\\0 &  \text{elsewhere}\end{cases}.\notag\end{equation}
\end{theorem}

\begin{proof}
We notice that, \begin{equation}\begin{aligned}
d\omega_{\mathbb{A}}(S)&=\sup\left\{\sqrt{|\langle Sz,z \rangle_{\mathbb{A}}|^2+\|Sz\|_{\mathbb{A}}^4}: z\in{\mathbb{H}}, \|z\|_{\mathbb{A}}=1\right\}\\&=\sup\left\{\sqrt{|\langle Sz,z \rangle_{\mathbb{A}}|^2+|\langle S^{*_\mathbb{A}}Sz,z \rangle_{\mathbb{A}}|^2}: z\in{\mathbb{H}}, \|z\|_{\mathbb{A}}=1\right\}\\&\le\sup\Big\{|\langle Sz,z \rangle_{\mathbb{A}}|: z\in{\mathbb{H}}, \|z\|_{\mathbb{A}}=1\Big\}+\sup\Big\{|\langle S^{*_\mathbb{A}}Sz,z \rangle_{\mathbb{A}}|: z\in{\mathbb{H}}, \|z\|_{\mathbb{A}}=1\Big\},\end{aligned}\notag\end{equation} where  the last inequality follows from the fact that $\sqrt{c+d}\le\sqrt{c}+\sqrt{d}$ holds for $c,d\ge0.$

Let $z =[z_1, z_2,  \cdots  ,z_n
]^{t}\in{\mathbb{H}}$ 
where $\|z\|_{\mathbb{A}}=1$
and
$\tilde{z}_{A}=\begin{bmatrix}
\|z_1\|_{A}, & \|z_2\|_{A}, & \cdots & ,\|z_n\|_{A}
\end{bmatrix}^{t}.$ 

Then it is clear that $\tilde{z}_{A}\in{\mathbb{C}^n}$ is a unit vector.
  We have, \begin{equation}\begin{aligned}|\langle Sz,z \rangle_{\mathbb{A}} |&= \left|\sum\limits_{i,j=1}^{n} \langle S_{ij}z_j,z_i \rangle_{A}\right|\\&\le \sum\limits_{i,j=1}^{n} \left|\langle S_{ij}z_j,z_i \rangle_{A}\right|\\&=\sum\limits_{i=1}^{n} \left|\langle S_{ii}z_i,z_i \rangle_{A}\right|+\sum\limits_{\substack{i,j=1\\i \ne j}}^{n} \left|\langle S_{ij}z_j,z_i \rangle_{A}\right|\\&=\sum\limits_{i=1}^{n} \left|\langle S_{ii}z_i,z_i \rangle_{A}\right|+\sum\limits_{\substack{i,j=1\\i < j}}^{n}\big(\hspace{0.05cm}\left|\langle S_{ij}z_j,z_i \rangle_{A}\right|+\left|\langle S_{ji}z_i,z_j \rangle_{A}\right|\hspace{0.05cm}\big)
  \end{aligned}\notag\end{equation}
  
 So,\begin{equation}\label{(35)}\begin{aligned}|\langle Sz,z \rangle_{\mathbb{A}} | &\le\sum\limits_{i=1}^{n} \omega_{A}( S_{ii})\|z_i\|^2_{A}+\sum\limits_{\substack{i,j=1\\i < j}}^{n}\big(\hspace{0.05cm}\| S_{ij} \|_{A}+\| S_{ji}\|_{A}\hspace{0.05cm}\big)\|z_i\|_{A}\|z_j\|_{A}.
\end{aligned}\tag{35}\end{equation}  

By the similar approach, we derive
\begin{equation}\begin{aligned}|\langle S^{*_A}Sz,z \rangle_{\mathbb{A}} |&= \left|\sum\limits_{i,j=1}^{n} \langle S^{*_A}_{ij}S_{ij}z_j,z _i\rangle_{A}\right|\\&\le \sum\limits_{i=1}^{n}\omega_{A}(S^{*_A}_{ii}S_{ii})\|z_{i}\|_{A}^2+\sum\limits_{\substack{i,j=1\\i < j}}^{n}\big(\hspace{0.05cm}\|S_{ij}^{*_A} S_{ij} \|_{A}+\|S_{ji}^{*_A} S_{ji}\|_{A}\hspace{0.05cm}\big)\|z_i\|_{A}\|z_j\|_{A}.
\end{aligned}\notag\end{equation} 
So,
\begin{equation}\label{(36)}\begin{aligned}|\langle S^{*_A}Sz,z \rangle_{\mathbb{A}} |\le \sum\limits_{i=1}^{n}\|S_{ii}\|_{A}^2\|z_{i}\|_{A}^2+\sum\limits_{\substack{i,j=1\\i < j}}^{n}\big(\hspace{0.05cm}\|S_{ij} \|^2_{A}+\|S_{ji}\|^2_{A}\hspace{0.05cm}\big)\|z_i\|_{A}\|z_j\|_{A}.
\end{aligned}\tag{36}\end{equation}

Adding (\ref{(35)}), (\ref{(36)}) and then considering supremum over $z\in{\mathbb{H}}$ with $\|z\|_{\mathbb{A}}=1,$ we get
\begin{equation}\begin{aligned}
d\omega_{\mathbb{A}}(S)&\le\sum\limits_{i,j=1}^n\widetilde{s_{ij}}\|z_i\|_{A}\hspace{0.06cm}\|z_j\|_{A}\\&=\langle\mathbb{\breve{S}} \tilde{z}_{A},\tilde{z}_{A}\rangle\\&=\omega(\mathbb{\breve{S}}),   
\end{aligned}\notag\end{equation} which completes the proof.
\end{proof}
As an immediate application of Theorem \ref{Theorem 3.7}, we see the following inequality. 
\begin{corollary}\label{Corollary 3.8}
Let $S_{ij}\in{\mathbb{B}_{A}(\mathcal{H})}$ for all $i,j=1,2.$
\begin{equation}\label{(37)}
d\omega_{\mathbb{A}}\left(\begin{bmatrix}
S_{11} & S_{12}  \\
S_{21} & S_{22} 
\end{bmatrix}\right)\le \frac{1}{2}\left(b_1+b_3+\sqrt{(b_1-b_3)^2+b_2^{2}}\right), 
\tag{37}\end{equation}
where 
\begin{equation*}\begin{aligned}
&b_1=\omega_{A}\left(S_{11}\right)+\|S_{11}\|_{A}^2, \hspace{1cm} b_3=\omega_{A}\left(S_{22}\right)+\|S_{22}\|_{A}^2,\\&b_2=\|S_{12} \|_{A}+\|S_{21}\|_{A}+\|S_{12} \|^2_{A}+\|S_{21}\|^2_{A}.\end{aligned}\end{equation*}
\end{corollary}

\begin{proof}
Letting  $b_1,b_2,b_3$ as described above and $n=2$ in   Theorem \ref{Theorem 3.7}, we get that
\begin{equation}\begin{aligned}d\omega_{\mathbb{A}}\left(\begin{bmatrix}
S_{11} & S_{12}  \\
S_{21} & S_{22} 
\end{bmatrix}\right)&\le\omega\left(\begin{bmatrix}
b_1 &  b_2 \\
 0& b_3 
\end{bmatrix}\right)\\&=r\left(\begin{bmatrix}
b_1 & \frac{b_2}{2} \\
 \frac{b_2}{2} & b_3 
\end{bmatrix}\right)\hspace{1cm}(\text{Lemma \ref{Lemma 3.3}})\\&=\frac{1}{2}\left(b_1+b_3+\sqrt{(b_1-b_3)^2+b_2^{2}}\right),\end{aligned}\notag\end{equation} which provides the desired inequality (\ref{(37)}).
\end{proof}

\begin{remark}
Considering $A=I$ in Corollary  \ref{Corollary 3.8}, the inequality (\ref{(37)}) becomes 
\begin{equation}
d\omega\left(\begin{bmatrix}
S_{11} & S_{12}  \\
S_{21} & S_{22} 
\end{bmatrix}\right)\le \frac{1}{2}\left(b_1+b_3+\sqrt{(b_1-b_3)^2+b_2^{2}}\right), 
\notag\end{equation}
where 
\begin{equation*}\begin{aligned}
&b_1=\omega\left(S_{11}\right)+\|S_{11}\|^2, \hspace{1cm} b_3=\omega\left(S_{22}\right)+\|S_{22}\|^2,\\&b_2=\|S_{12} \|+\|S_{21}\|+\|S_{12} \|^2+\|S_{21}\|^2.\end{aligned}\end{equation*}
 Thus, the inequality (\ref{(37)}) generalizes  the inequality which is given in [\citenum{alomari2023davis}, Corollary 3.1].
\end{remark}

Active research is underway to extend standard numerical radius bounds to the $A-$numerical radius.
We can draw more attention about the generalized $A-$numerical radius when we use this theory with positive operators and functions that are continuous.

The following finding  is an obvious outcome of the ongoing improvements in $A-$weighted numerical radius theory and block operator approaches.

\begin{theorem} \label{Theorem 3.10}

Let $X_{j},Y_j,C_{jk},D_{jk}$  for $j,k=1,2,\cdots,n$  be operators in $\mathbb{B}_{A}(\mathcal{H})$ such that $X_j, Y_j$ are $A-$positive  operators, and let $f$ and $g$ be two non-negative continuous functions on $[0,\infty).$ Then 
\begin{equation}\label{(38)}
\mbox{ \footnotesize$\omega_{A}\left(\sum\limits_{j,k=1}^n \Big(f(X_j)C_{jk}g(Y_k)+g(Y_j)D_{jk}f(X_k)\Big)\right)  \le\frac{1}{2}\left(\|\mathbb{P}\|_{\mathbb{A}}+\|\mathbb{Q}\|_{\mathbb{A}}\right) \left\|\sum\limits_{j=1}^n f^2(X_j)+g^2(Y_j)\right\|_{A},$}
\tag{38}\end{equation}
where 

$\mathbb{P}=\begin{bmatrix}
C_{11} & C_{12} & \cdots & C_{1n} \\
C_{21} & C_{22} & \cdots & C_{2n} \\
\vdots & \vdots & \ddots & \vdots \\
C_{n1} & C_{n2} & \cdots & C_{nn}
\end{bmatrix}$  and $\mathbb{Q}=\begin{bmatrix}
D_{11} & D_{12} & \cdots & D_{1n} \\
D_{21} & D_{22} & \cdots & D_{2n} \\
\vdots & \vdots & \ddots & \vdots \\
D_{n1} & D_{n2} & \cdots & D_{nn}
\end{bmatrix}\in{\mathbb{B}_{\mathbb{A}}(\mathbb{H}}).$ 

\end{theorem}

\begin{proof}
Let $\mathbb{M}=\begin{bmatrix}
f(X_1)  & 0 & \cdots & 0 \\
f(X_2)  & 0 & \cdots & 0 \\
\vdots   & \vdots & \ddots & \vdots \\
f(X_n)  & 0 & \cdots & 0 
\end{bmatrix},$
$\mathbb{N}=\begin{bmatrix}
g(Y_1)  & 0 & \cdots & 0 \\
g(Y_2) & 0 & \cdots & 0 \\
\vdots  & \vdots & \ddots & \vdots \\
g(Y_n)  & 0 & \cdots & 0 
\end{bmatrix},$ 

\begin{align*}\mathbb{P}=\begin{bmatrix}
C_{11} & C_{12} & \cdots & C_{1n} \\
C_{21} & C_{22} & \cdots & C_{2n} \\
\vdots & \vdots & \ddots & \vdots \\
C_{n1} & C_{n2} & \cdots & C_{nn}
\end{bmatrix}  
\text{and}\hspace{0.1cm}
\mathbb{Q}=\begin{bmatrix}
D_{11} & D_{12} & \cdots & D_{1n} \\
D_{21} & D_{22} & \cdots & D_{2n} \\
\vdots & \vdots & \ddots & \vdots \\
D_{n1} & D_{n2} & \cdots & D_{nn}
\end{bmatrix}\in{\mathbb{B}_{\mathbb{A}}(\mathbb{H}}).\end{align*}

Now, it is not difficult to verify that

$$\mathbb{M}^{*_{\mathbb{A}}}\mathbb{P}\mathbb{N}+\mathbb{N}^{*_{\mathbb{A}}}\mathbb{Q}\mathbb{M}=\begin{bmatrix}
\sum\limits_{j,k=1}^n f(X_j)C_{jk}g(Y_k)+g(Y_j)D_{jk}f(X_k)  & 0 & \cdots & 0 \\
0  & 0 & \cdots & 0 \\
\vdots  & \vdots & \ddots & \vdots \\
0  & 0 & \cdots & 0 
\end{bmatrix}.$$

Let $z\in{\mathbb{H}}$ where $\|z\|_{\mathbb{A}}=1.$

Utilizing Cauchy-Schwarz inequality and  AM-GM inequality, we see that,\begin{equation}\begin{aligned}&\left|\left\langle \left(\sum\limits_{j,k=1}^n f(X_j)C_{jk}g(Y_k)+g(Y_j)D_{jk}f(X_k)\right)z,z\right\rangle_{A}\right|\\&=\left|\left\langle \left(\mathbb{M}^{*_{\mathbb{A}}}\mathbb{P}\mathbb{N}+\mathbb{N}^{*_{\mathbb{A}}}\mathbb{Q}\mathbb{M}\right)z,z\right\rangle_{\mathbb{A}}\right|\\&\le\left|\left\langle \mathbb{M}^{*_{\mathbb{A}}}\mathbb{P}\mathbb{N}z,z\right\rangle_{\mathbb{A}}\right|+\left|\left\langle \mathbb{N}^{*_{\mathbb{A}}}\mathbb{Q}\mathbb{M}z,z\right\rangle_{\mathbb{A}}\right|\\&=\left|\left\langle \mathbb{P}\mathbb{N}z,\mathbb{M}z\right\rangle_{\mathbb{A}}\right|+\left|\left\langle \mathbb{Q}\mathbb{M}z,\mathbb{N}z\right\rangle_{\mathbb{A}}\right|\\&\le\|\mathbb{PN}z\|_{\mathbb{A}}\|\mathbb{M}z\|_{\mathbb{A}}+\|\mathbb{QM}z\|_{\mathbb{A}}\|\mathbb{N}z\|_{\mathbb{A}}\\&\le\frac{1}{2}\left(\|\mathbb{P}\|_{\mathbb{A}}+\|\mathbb{Q}\|_{\mathbb{A}}\right)\hspace{0.04cm}\left\langle\mathbb{(N^{*_A}N+M^{*_A}M)}z,z\right\rangle_{\mathbb{A}}
\end{aligned}\notag\end{equation}
So,\begin{equation}\label{(39)}\mbox{\footnotesize$ \left|\left\langle \left(\sum\limits_{j,k=1}^n f(X_j)C_{jk}g(Y_k)+g(Y_j)D_{jk}f(X_k)\right)z,z\right\rangle_{A}\right|\le\frac{1}{2}\left(\|\mathbb{P}\|_{\mathbb{A}}+\|\mathbb{Q}\|_{\mathbb{A}}\right)\hspace{0.04cm}\left\|\mathbb{N^{*_A}N+M^{*_A}M}\right\|_{\mathbb{A}}$}\tag{39}\end{equation}
Now, $\left\|\mathbb{N^{*_A}N+M^{*_A}M}\right\|_{\mathbb{A}}=\left\|\begin{bmatrix}
\sum\limits_{j=1}^n f^2(X_j)+g^2(Y_j) & 0        & 0      & \cdots & 0 \\
0      & 0        & 0      & \cdots & 0 \\
\vdots & \vdots   & \vdots & \ddots & \vdots \\
0      & 0        & 0      & \cdots & 0
\end{bmatrix}\right\|_{\mathbb{A}}.$

Using the fact $\left\|\begin{bmatrix}
Z_1 & 0   & \cdots & 0 \\
0   & Z_2 & \cdots & 0 \\
\vdots & \vdots & \ddots & \vdots \\
0   & 0   & \cdots & Z_n
\end{bmatrix}\right\|_{
\mathbb{A}
}=\max\limits_{1\le j\le n}\{\|Z_j\|_{A}\}$ for every $Z_j\in{\mathbb{B}_{A}(\mathcal{H}}),$ $j=1,2,\cdots,n,$ and taking supremum over $z\in{\mathbb{H}}$ with $\|z\|_{\mathbb{A}}=1$ on the inequality (\ref{(39)}), we obtain
the desired inequality (\ref{(38)}).
\end{proof}

If we loose the condition of $A-$ positiveness of $X_j,Y_j$ for all $j=1,2,\cdots,n$, and we choose $\mathbb{M}=\begin{bmatrix}
X_1  & 0 & \cdots & 0 \\
X_2  & 0 & \cdots & 0 \\
\vdots   & \vdots & \ddots & \vdots \\
X_n  & 0 & \cdots & 0 \\
\end{bmatrix},$   $\mathbb{N}=\begin{bmatrix}
Y_1  & 0 & \cdots & 0 \\
Y_2 & 0 & \cdots & 0 \\
\vdots  & \vdots & \ddots & \vdots \\
Y_n  & 0 & \cdots & 0 \\
\end{bmatrix}$ for the particular choice of two functions $f,g$ in the proof of Theorem \ref{Theorem 3.10} then Theorem  \ref{Theorem 3.10} is stated as follows:

\begin{theorem} \label{Theorem 3.11}
 Let $X_{j},Y_j,C_{jk},D_{jk}$  for $j,k=1,2,\cdots,n$  be operators in $\mathbb{B}_{A}(\mathcal{H})$.
 Then 
\begin{equation}\label{(40)}
\mbox{ \footnotesize$\omega_{A}\left(\sum\limits_{j,k=1}^n \Big(X_j^{*_A}C_{jk}Y_k+Y_j^{*_A}D_{jk}X_k\Big)\right)  \le \frac{1}{2} \left(\|\mathbb{P}\|_{\mathbb{A}}+ \|\mathbb{Q}\|_{\mathbb{A}}\right) \left\|\sum\limits_{j=1}^n  X_j^{*_A}X_j+Y_j^{*_A}Y_j\right\|_{A},$}
\tag{40}\end{equation}
where

$\mathbb{P}=\begin{bmatrix}
C_{11} & C_{12} & \cdots & C_{1n} \\
C_{21} & C_{22} & \cdots & C_{2n} \\
\vdots & \vdots & \ddots & \vdots \\
C_{n1} & C_{n2} & \cdots & C_{nn}
\end{bmatrix}$ and $\mathbb{Q}=\begin{bmatrix}
D_{11} & D_{12} & \cdots & D_{1n} \\
D_{21} & D_{22} & \cdots & D_{2n} \\
\vdots & \vdots & \ddots & \vdots \\
D_{n1} & D_{n2} & \cdots & D_{nn}
\end{bmatrix}\in{\mathbb{B}_{\mathbb{A}}(\mathbb{H}}).$ 
   
\end{theorem}

If we choose $C_{jk}=D_{jk}=O$ for all $j\ne k$, and we take $C_{jj}=D_{jj}=I$ for all $j=1,2,\cdots,n$ in Theorem \ref{Theorem 3.11} we obtain the following bound.

\begin{corollary}\label{Corollary 3.12}
  Let $X_{j},Y_j$  for $j=1,2,\cdots,n$  be operators in $\mathbb{B}_{A}(\mathcal{H})$. \begin{equation}\label{(41)}
\omega_{A}\left(\sum\limits_{j=1}^n \Big(X_j^{*_A}Y_j+Y_j^{*_A}X_j\Big)\right)  \le  \left\|\sum\limits_{j=1}^n  X_j^{*_A}X_j+Y_j^{*_A}Y_j\right\|_{A}.
\tag{41}\end{equation}   
\end{corollary}

\begin{corollary}
  Let $X_{j},Y_j$ be operators in $\mathbb{B}_{A}(\mathcal{H})$ such that $X_j^{*_A}Y_j=Y_j^{*_A}X_j$ for $j=1,2,\cdots,n.$ Then
  
  \begin{equation}\label{(42)}
\omega_{A}\left(\sum\limits_{j=1}^n X_j^{*_A}Y_j\right)  \le
\frac{1}{2}\left\|\sum\limits_{j=1}^n  X_j^{*_A}X_j+Y_j^{*_A}Y_j\right\|_{A}.
\tag{42}\end{equation}   
\end{corollary}
\begin{remark}
(i) It is not difficult to show that $X_j^{*_A}Y_j+Y_j^{*_A}X_j$ is $A-$ self-adjoint.
It is well-known that $\sup\limits_{\theta\in{\mathbb{R}}}\|\Re_{A}(e^{i\theta}Z)\|_{A}=\omega_{A}(Z),$ where $Z\in{\mathbb{B}_{A}(\mathcal{H})}.$
From Corollary \ref{Corollary 3.12}, we see that for every $\theta\in{\mathbb{R}},$
\begin{equation}
 \left\|\sum\limits_{j=1}^n \Re_{A}(e^{i\theta}X_j^{*_A}Y_j)\right\|_{A} \le   \frac{1}{2}\left\|\sum\limits_{j=1}^n  X_j^{*_A}X_j+Y_j^{*_A}Y_j\right\|_{A}
\notag\end{equation}
i.e. \begin{equation}\label{(43)}
 \left\|\Re_{A}(e^{i\theta}\sum\limits_{j=1}^n X_j^{*_A}Y_j)\right\|_{A} \le   \frac{1}{2}\left\|\sum\limits_{j=1}^n  X_j^{*_A}X_j+Y_j^{*_A}Y_j\right\|_{A}
\tag{43}\end{equation}
Taking supremum over all $\theta\in{\mathbb{R}}$, we establish that

\begin{equation*}
    \omega_{A}\left(\sum\limits_{j=1}^n X_j^{*_A}Y_j\right)\le\frac{1}{2}\left\| X_j^{*_A}X_j+Y_j^{*_A}Y_j\right\|_{A}.\end{equation*}

By letting $n=1$ and $X,Y\in{\mathbb{B}_{A}(\mathcal{H})}$, we obtain that
\begin{equation*}
\omega_{A}\left(X^{*_A}Y\right)  \le \frac{1}{2} \left\| X^{*_A}X+Y^{*_A}Y\right\|_{A},
\end{equation*}
which is proved as a special case in [\citenum{guesba2021some}, Theorem 2.1].

(ii) Discovering Theorem \ref{Theorem 3.10} yields a significant general form that makes a connection between the operator-matrix norm with the sums of the corresponding functions of $X_j,Y_j$ and the $A-$numerical radius $\omega_{A}$ of the sums of operator products with a matrix structure $f(X_j)C_{jk}g(Y_k)+g(Y_j)D_{jk}f(X_k).$
It is a result that brings everything together and describes how operators work in terms of their functional changes.

(iii) By considering $D_{jk}=O$  for $j,k=1,2,\cdots,n$ in Theorem \ref{Theorem 3.11} we get that 
\begin{equation*}
\mbox{ \footnotesize$\omega_{A}\left(\sum\limits_{j,k=1}^n X_j^{*_A}C_{jk}Y_k\right)  \le \frac{1}{2} \|\mathbb{P}\|_{\mathbb{A}} \left\|\sum\limits_{j=1}^n  X_j^{*_A}X_j+Y_j^{*_A}Y_j\right\|_{A},$}
\end{equation*}
where $\mathbb{P}=\begin{bmatrix}
C_{11} & C_{12} & \cdots & C_{1n} \\
C_{21} & C_{22} & \cdots & C_{2n} \\
\vdots & \vdots & \ddots & \vdots \\
C_{n1} & C_{n2} & \cdots & C_{nn}
\end{bmatrix},$ which generalizes the estimation given in [\citenum{bhunia2021new}, Theorem 2.19]. Consequently, we can conclude that  Theorem \ref{Theorem 3.10} allows more flezibility of the estimation proved in [\citenum{bhunia2021new}, Theorem 2.19] because of the use of two non-negative continuous functions.
\end{remark}

\section{Upper bounds on the 
\texorpdfstring{$A-$s}spectral radius of sums of operator products and related 
\texorpdfstring{$A-$n}numerical radius inequalities}\label{sec4}

We present the following result which shows the upper bound of $A-$spectral radius for the sum of product of operators.

\begin{theorem}\label{Theorem 4.1}
Let $C_1,C_2,D_1,D_2,X_1,X_2$ $\in{\mathbb{B}_{A}(\mathcal{H})}.$ Then we have
\begin{equation}\begin{aligned}\label{(44)}r_{A}(C_1X_1D_1+C_2X_2D_2)\le \min(\alpha_1,\alpha_2),\end{aligned}\tag{44}\end{equation} where
\begin{equation}\begin{aligned}
\mbox{\footnotesize $\alpha_{1}=\frac{1}{2}\bigg(\|X_1D_1C_1\|_{A}+\|D_2C_2X_2\|_{A}+\sqrt{(\|X_1D_1C_1\|_{A}-\|D_2C_2X_2\|_{A})^2+4\|X_1D_1C_2X_2\|_{A}\|D_2C_1\|_{A}}\bigg)$} \notag\end{aligned}\end{equation}and
\begin{equation}\begin{aligned}\mbox{\footnotesize $\alpha_{2}=\frac{1}{2}\bigg(\|D_1C_1X_1\|_{A}+\|X_2D_2C_2\|_{A}+\sqrt{(\|D_1C_1X_1\|_{A}-\|X_2D_2C_2\|_{A})^2+4\|X_2D_2C_1X_1\|_{A}\|D_1C_2\|_{A}}\bigg) $}
.\notag\end{aligned}\end{equation}\end{theorem}
\begin{proof}
By putting $A_1=C_1, B_1=X_1D_1,A_2=C_2X_2$ and $B_2=D_2$ in Lemma \ref{Lemma 2.1}, we infer that
\begin{equation}\label{(45)}
 r_{A}(C_1X_1D_1+C_2X_2D_2)\le \alpha_1.   
\tag{45}\end{equation}
Again by setting $A_1=C_1X_1, B_1=D_1,A_2=C_2$ and $B_2=X_2D_2$ in Lemma \ref{Lemma 2.1}, we infer that
\begin{equation}\label{(46)}
 r_{A}(C_1X_1D_1+C_2X_2D_2)\le \alpha_2.   
\tag{46}\end{equation}
Combining (\ref{(45)}) and (\ref{(46)}), we conclude the desired inequality (\ref{(44)}).
\end{proof}

\begin{corollary}\label{Corollary 4.2}
 If $S\in {\mathbb{B}_{A}(\mathcal{H})}$ with closed range then we have
 \begin{equation}\label{(47)}
  \begin{aligned}
r_{A}(S)\le\frac{1}{4}\bigg(\sqrt{\Big(\|S^2S^{\dag}\|_{A}-\|S^{\dag}S^2\|_{A}\Big)^2+4\|S^2\|_{A}\|S^{\dag}S^2S^{\dag}\|_{A}}+\|S^2S^{\dag}\|_{A}+\|S^{\dag}S^2\|_{A}\bigg).   
 \end{aligned}  \tag{47} 
 \end{equation}
\end{corollary}

\begin{proof}
If $S^{\dag}$  is the Moore-Penrose of $S$ then we see that $S=SS^{\dag}S.$
Now, by choosing $C_1=C_2=S, X_1=X_2=S^{\dag}$ and $D_1=D_2=S$ in 
 Theorem \ref{Theorem 4.1}, we get    
the desired inequality (\ref{(47)}).
\end{proof}
\begin{remark}
Considering $A=I$  in Corollary \ref{Corollary 4.2}, we observe that
\begin{equation}
\begin{aligned}
&r(S)\le\frac{1}{4}\bigg(\sqrt{\Big(\|S^2S^{\dag}\|-\|S^{\dag}S^2\|\Big)^2+4\hspace{0.05cm}\|S^2\|\hspace{0.1cm}\|S^{\dag}S^2S^{\dag}\|}+\|S^2S^{\dag}\|+\|S^{\dag}S^2\|\bigg)\\&\hspace{0.6cm}\le\frac{1}{4}\bigg(\|S(SS^{\dag})\|+\|(S^{\dag}S)S\| +\left|\|S(SS^{\dag})\|-\|(S^{\dag}S)S\|\right|\bigg)+\frac{1}{2}\sqrt{\|S^2\|\hspace{0.1cm}\|(S^{\dag}S)(SS^{\dag})\|}\\&\hspace{0.6cm}=\frac{1}{2} \max\Big\{\|S(SS^{\dag})\|,\|(S^{\dag}S)S\|\Big\} +\frac{1}{2}\sqrt{\|S^2\|\hspace{0.1cm}\|(S^{\dag}S)(SS^{\dag})\|}\\&\hspace{0.6cm}\le\frac{1}{2} \max\Big\{\|S\|\hspace{0.1cm}\|SS^{\dag}\|,\|S^{\dag}S\|\hspace{0.1cm}\|S\|\Big\} +\frac{1}{2}\sqrt{\|S^2\|} \sqrt{\|S^{\dag}S\|\hspace{0.1cm}\|SS^{\dag}\|}.
\end{aligned}\notag   
\end{equation}
Since $\|S^{\dag}S\|\le1, \|SS^{\dag}\|\le1$ for any $S\in{\mathbb{CR}(\mathcal{H})},$ so we finally reach 
\begin{align*}
r(S)\le \frac{1}{2}\bigg(\|S\|+\sqrt{\|S^2\|}\bigg)    .\end{align*}
Hence, the inequality (\ref{(47)}) generalizes and refines the well-known inequality (\ref{(3)}).
\end{remark}

\begin{theorem}\label{Theorem 4.4}
Let $S\in {\mathbb{B}_{A}(\mathcal{H})},$ and let $S^{*_A}=P+iQ$ be the Cartesian decomposition of $S^{*_A},$ where $P=\frac{T_{1}^{*_{A}}+(T_{1}^{*_{A}})^{*_{A}}}{2}=\Re{T_{1}^{*_{A}}}$ and $Q=\frac{T_{1}^{*_{A}}-(T_{1}^{*_{A}})^{*_{A}}}{2i}=\Im{T_{1}^{*_{A}}}.$ Then we have  
\begin{equation}\label{(48)}
(1)\hspace{0.1cm}\omega^2_{A}(S)\ge\frac{1}{2}\max\Big\{\|P+Q\|^2_{A}+  c^2_{A}(P-Q),\|P-Q\|^2_{A}+  c^2_{A}(P+Q)\Big\},  \tag{48}\end{equation}
\begin{equation}\label{(49)}
(2)\hspace{0.1cm}\omega^2_{A}(S)\ge\max\Big\{\|P+Q\|_{A}\hspace{0.1cm}c_{A}(P-Q),\|P-Q\|_{A}\hspace{0.1cm} c_{A}(P+Q)\Big\},  \tag{49}\end{equation}

\begin{equation}\label{(50)}
\begin{aligned}&(3)\hspace{0.1cm}\omega^2_{A}(S)\ge\frac{1}{4}\|S^{*_A}(S^{*_A})^{*_A}+(S^{*_A})^{*_A}S^{*_A}\|_{A} + \frac{1}{4}\Big(c^2_{A}(P-Q)+  c^2_{A}(P+Q)\Big)\\&\hspace{1.4cm}+\frac{1}{4}\left|\|P+Q\|^2_{A}-\|P-Q\|^2_{A}+c^2_{A}(P-Q)-  c^2_{A}(P+Q)\right|.\end{aligned}\tag{50}\end{equation}
\end{theorem}
\begin{proof}
Consider $S^{*_A}=P+iQ$ and $z\in{\mathcal{H}}$ where $\|z\|_{A}=1$. 
Clearly it is observed that $P=\frac{S^{*_A}+(S^{*_A})^{*_A}}{2}$ and $Q=\frac{S^{*_A}-(S^{*_A})^{*_A}}{2i}.$
Now, the simple calculation shows that
\begin{equation}\label{(51)}
\begin{aligned}&|\langle (P+Q)z,z \rangle_{A}|^2+  |\langle (P-Q)z,z \rangle_{A}|^2
\\&=\langle (P+Q)z,z \rangle_{A}\hspace{0.1cm} \langle z,(P+Q)z \rangle_{A}+\langle (P-Q)z,z \rangle_{A}\hspace{0.1cm}\langle z, (P-Q)z\rangle_{A}\\&=2\hspace{0.06cm}\left(|\langle Pz,z \rangle_{A}|^2+|\langle Qz,z \rangle_{A}|^2\right).
\end{aligned} \tag{51} \end{equation} 
We have,
\begin{equation}\label{(52)}
\begin{aligned}
  &|\langle S^{*_A}z,z \rangle_{A}|^2 =|\langle Pz,z \rangle_{A}+i\langle Qz,z \rangle_{A}|^2\\&\hspace{1.9cm}=|\langle Pz,z \rangle_{A}|^2+ |\langle Qz,z \rangle_{A}|^2\\&\hspace{1.9cm}=\frac{1}{2}\Big(|\langle (P+Q)z,z \rangle_{A}|^2+  |\langle (P-Q)z,z \rangle_{A}|^2\Big)\hspace{0.2cm}\text{ (by (\ref{(51)}))}\\&\hspace{1.9cm}\ge\frac{1}{2}\Big(|\langle (P+Q)z,z \rangle_{A}|^2+  c^2_{A}(P-Q)\Big)\\&\hspace{1.9cm}\ge\frac{1}{2}\Big(\|P+Q\|^2_{A}+  c^2_{A}(P-Q)\Big).
\end{aligned} \tag{52} 
\end{equation}
By deciding on the supremum over $z\in{\mathcal{H}}$ where $\|z\|_{A} = 1,$  we infer that
\begin{equation}\label{(53)}
\omega^2_{A}(S)= \omega^2_{A}(S^{*_A}) \ge\frac{1}{2}\Big(\|P+Q\|^2_{A}+  c^2_{A}(P-Q)\Big).
\tag{53}\end{equation}
Again, from  (\ref{(51)}), we get
\begin{equation*}
 |\langle S^{*_A}z,z \rangle_{A}|^2 \ge\frac{1}{2}\Big(|\langle (P-Q)z,z \rangle_{A}|^2+  c^2_{A}(P+Q)\Big)   
\end{equation*}
By deciding on the supremum over $z\in{\mathcal{H}}$ where $\|z\|_{A} = 1,$  we infer that
\begin{equation}\label{(54)}
\omega^2_{A}(S)= \omega^2_{A}(S^{*_A}) \ge\frac{1}{2}\Big(\|P-Q\|^2_{A}+  c^2_{A}(P+Q)\Big).
\tag{54}\end{equation}
Combining (\ref{(53)}) and (\ref{(54)}), we reach the inequality (\ref{(48)}).

Now, applying AM-GM inequality in (\ref{(52)}), we notice that
\begin{equation}
\begin{aligned}&
|\langle S^{*_A}z,z \rangle_{A}|^2\ge|\langle (P+Q)z,z \rangle_{A}|\hspace{0.2cm}  |\langle (P-Q)z,z \rangle_{A}|\\&\hspace{1.9cm}\ge |\langle (P+Q)z,z \rangle_{A}|\hspace{0.1cm} c_{A}(P-Q)
\end{aligned}\notag   
\end{equation}
By deciding on the supremum over $z\in{\mathcal{H}}$ where $\|z\|_{A} = 1,$  we get that
\begin{equation}\label{(55)}
\omega^2_{A}(S)\ge\|P+Q\|_{A}\hspace{0.1cm}c_{A}(P-Q)
\tag{55}\end{equation}
Similarly, we obtain
\begin{equation}\label{(56)}
\omega^2_{A}(S)\ge\|P-Q\|_{A}\hspace{0.1cm}c_{A}(P+Q)
\tag{56}\end{equation}
Combining (\ref{(55)}) and (\ref{(56)}), we reach the inequality (\ref{(49)}).

From (\ref{(48)}), we have
\begin{equation}\begin{aligned}&\omega^2_{A}(S)\ge\frac{1}{2}\max\Big\{\|P+Q\|^2_{A}+  c^2_{A}(P-Q),\|P-Q\|^2_{A}+  c^2_{A}(P+Q)\Big\}\\&\hspace{0.8cm}=\frac{1}{4}\left(\|P+Q\|^2_{A}+ \|P-Q\|^2_{A}+c^2_{A}(P-Q)+  c^2_{A}(P+Q)\right)+\\&\hspace{1.2cm}\frac{1}{4}\left|\|P+Q\|^2_{A}-\|P-Q\|^2_{A}+c^2_{A}(P-Q)-  c^2_{A}(P+Q)\right|\\&\hspace{0.8cm}\ge\frac{1}{4}\left(\|(P+Q)^2\|_{A}+ \|(P-Q)^2\|_{A}+c^2_{A}(P-Q)+  c^2_{A}(P+Q)\right)+\\&\hspace{1.2cm}\frac{1}{4}\left|\|P+Q\|^2_{A}-\|P-Q\|^2_{A}+c^2_{A}(P-Q)-  c^2_{A}(P+Q)\right|\\&\hspace{0.8cm}\ge\frac{1}{2}\hspace{0.07cm}\|P^2+Q^2\|_{A}+\frac{1}{4}\Big(c^2_{A}(P-Q)+  c^2_{A}(P+Q)\Big)+\\&\hspace{1.2cm}\frac{1}{4}\left|\|P+Q\|^2_{A}-\|P-Q\|^2_{A}+c^2_{A}(P-Q)-  c^2_{A}(P+Q)\right|
\end{aligned}\notag\end{equation}
Observe that, 
\begin{equation*}
P^2+Q^2=\frac{1}{2}\Big(S^{*_A}(S^{*_A})^{*_A}+(S^{*_A})^{*_A}S^{*_A}\Big)
\end{equation*}
Utilizing the fact $\|S\|_{A}=\|S^{*_A}\|_{A}$ for $S\in {\mathbb{B}_{A}(\mathcal{H})},$ we reach the inequality (\ref{(50)}).
\end{proof}
Based on Theorem \ref{Theorem 4.4}, we prove the following inequality.
\begin{theorem}\label{Theorem 4.5}
 Let $T_1,T_2,z,Y\in {\mathbb{B}_{A}(\mathcal{H})}.$   Then  
 \begin{equation}\label{(57)}
\begin{aligned}
\omega_{A}(T_1zT_2\pm T_2YT_1)\le 2\sqrt{2} \|T_2\|\max\{\|z\|_{A}, \|Y\|_{A}\} \times \alpha,  
\end{aligned} \tag{57}  
 \end{equation}
 where
\begin{equation}
\begin{aligned}
\alpha=&\Bigg[\omega_{A}^2(T_1)-\frac{1}{4}\Bigg(c_{A}^2(\Re{T_{1}^{*_{A}}}-\Im{T_{1}^{*_{A}}})+c_{A}^2(\Re{T_{1}^{*_{A}}}+\Im{T_{1}^{*_{A}}})  +\Big|\|\Re{T_{1}^{*_{A}}}+\Im{T_{1}^{*_{A}}}\|_{A}^2-\\&\|\Re{T_{1}^{*_{A}}}-\Im{T_{1}^{*_{A}}}\|_{A}^2+c_{A}^2(\Re{T_{1}^{*_{A}}}-\Im{T_{1}^{*_{A}}})-c_{A}^2(\Re{T_{1}^{*_{A}}}+\Im{T_{1}^{*_{A}}})\Big|\Bigg)\Bigg]^{\frac{1}{2}}, 
\end{aligned}\notag
\end{equation}

$\Re{T_{1}^{*_{A}}}=\frac{T_{1}^{*_{A}}+(T_{1}^{*_{A}})^{*_{A}}}{2}$ and $\Im{T_{1}^{*_{A}}}=\frac{T_{1}^{*_{A}}-(T_{1}^{*_{A}})^{*_{A}}}{2i}.$
\end{theorem}

\begin{proof}
Let us assume that $\|X\|_{A}\le1$ and  $\|Y\|_{A}\le1.$ Let $z\in{\mathcal{H}}$ where $\|z\|_{A}=1.$

Then we have
\begin{equation*}
\begin{aligned}
|\langle (T_1X\pm YT_1)z,z \rangle_{A}|&\le |\langle T_1X z,z \rangle_{A}|+ |\langle YT_1 z,z \rangle_{A}| \\&= |\langle X z,T_{1}^{*_{A}}z \rangle_{A}|+ |\langle T_1 z,Y^{*_{A}}z \rangle_{A}| \\&\le \|T_{1}^{*_{A}}z\|_{A} + \|T_1z\|_{A}\\&\le\sqrt{2}\sqrt{\|T_{1}^{*_{A}}z\|^2_{A} + \|T_1z\|^2_{A}}\\&=\sqrt{2}\sqrt{\langle (T_{1}^{*_{A}})^{*_{A}}T_{1}^{*_{A}}z,z \rangle_{A}+\langle T_{1}^{*_{A}} T_1 z,z \rangle_{A} }\\&\le\sqrt{2}\|(T_{1}^{*_{A}})^{*_{A}}T_{1}^{*_{A}}+T_{1}^{*_{A}} T_1\|_{A}^{\frac{1}{2}}\\&=\sqrt{2}\|(T_{1}^{*_{A}})^{*_{A}}T_{1}^{*_{A}}+T_{1}^{*_{A}}(T_{1}^{*_{A}})^{*_{A}}\|_{A}^{\frac{1}{2}}\\&\text{ (since $\|S^{*_{A}}\|_{A}=\|S\|_{A}$ and $((S^{*_{A}})^{*_{A}})^{*_{A}}=S^{*_{A}}$ for any $S\in {\mathbb{B}_{A}(\mathcal{H})} )$}\\&\le 2\sqrt{2}\alpha \text{\hspace{0.7cm}( by (\ref{(50)})  )}.
\end{aligned}   
\end{equation*}

By considering the supremum over $\|z\|_{A}=1,$ we get 
\begin{equation}\label{(58)}
\omega_{A}(T_1X\pm YT_{1})  \le 2\sqrt{2}\alpha.  
\tag{58}\end{equation}
For the general case, we choose $X ,Y$ as arbitrary operators. If $\max\{\|X \|_{A},\|Y\|_{A}\}=0$ then Theorem holds trivially. If $\max\{\|X\|_{A},\|Y\|_{A}\}\ne0$ then it is obivious that $\left\|\frac{X }
{\max\{\|X\|_{A},\|Y\|_{A}\}}\right\|_{A}\le 1$ and $\left\|\frac{Y}{\max\{\|X \|_{A},\|Y\|_{A}\}}\right\|_{A}\le 1.$

By setting $X=\frac{X}{\max\{\|X\|_{A},\|Y\|_{A}\}}$ and $Y=\frac{Y}{\max\{\|X\|_{A},\|Y\|_{A}\}}$ in (\ref{(58)}), respectively we arrive at the following inequality:
\begin{equation*}
 \omega_{A}(T_1X\pm YT_{1})  \le 2\sqrt{2} \max\{\|X\|_{A},\|Y\|_{A}\}\alpha.   
\end{equation*}
Now, substituting $X$ by $XT_{2}$ and $Y$ by $T_{2}Y$ in the above inequality, we finally obtain the required inequality (\ref{(57)}).
\end{proof}

The next corollary directly comes from  Theorem \ref{Theorem 4.5}.

\begin{corollary}\label{Corollary 4.6}
Let $T_1,T_2\in {\mathbb{B}_{A}(\mathcal{H})}.$ Then
\begin{equation}\label{(59)}
 \omega_{A}(T_1T_2\pm T_2T_1)\le 2\sqrt{2}\min\{h_{A}(T_1,T_2),h_{A}(T_2,T_1)\}   
\tag{59}\end{equation}
where
$h_{A}(T_1,T_2)=\|T_2\|_{A}\alpha$
and \begin{equation}
\begin{aligned}
\alpha=&\Bigg[\omega_{A}^2(T_1)-\frac{1}{4}\Bigg(c_{A}^2(\Re{T_{1}^{*_{A}}}-\Im{T_{1}^{*_{A}}})+c_{A}^2(\Re{T_{1}^{*_{A}}}+\Im{T_{1}^{*_{A}}})  +\Big|\|\Re{T_{1}^{*_{A}}}+\Im{T_{1}^{*_{A}}}\|_{A}^2-\\&\|\Re{T_{1}^{*_{A}}}-\Im{T_{1}^{*_{A}}}\|_{A}^2+c_{A}^2(\Re{T_{1}^{*_{A}}}-\Im{T_{1}^{*_{A}}})-c_{A}^2(\Re{T_{1}^{*_{A}}}+\Im{T_{1}^{*_{A}}})\Big|\Bigg)\Bigg]^{\frac{1}{2}}. \end{aligned}\notag
\end{equation}\end{corollary}

\begin{proof}
By considering $X=Y=I$ in Theorem  , we get
\begin{equation}\label{(60)}
  \omega_{A}(T_1T_2\pm T_2T_1)\le 2\sqrt{2} h_{A}(T_1,T_2).  
\tag{60}\end{equation}
Now by interchanging $T_1$ and $T_2$ in  , we get
\begin{equation}\label{(61)}
  \omega_{A}(T_1T_2\pm T_2T_1)\le 2\sqrt{2} h_{A}(T_2,T_1).  
\tag{61}\end{equation}
Combining the inequalities (\ref{(60)}) and  (\ref{(61)}) we arrive at the desired inequality (\ref{(59)}).
\end{proof}

\begin{remark}
(a) The Corollary \ref{Corollary 4.6}  refines the inequality which is proved in [\citenum{zamani2019numerical}, Theorem 4.2] and [\citenum{feki2021some}, Corollary 2.17], namely

\begin{equation*}
 \omega_{A}(T_1T_2\pm T_2T_1)\le 2\sqrt{2}\min\left\{\|T_1\|_{A}\omega_{A}(T_2), \|T_2\|_{A}\omega_{A}(T_1) \right\}.   
\end{equation*}
(b) The inequality (\ref{(50)}) improves the lower bound of (\ref{(4)}).
\end{remark}

\textbf{Funding} The first author expresses gratitude to the University Grants Commission (UGC), Government of India, for providing financial assistance in the form of senior research fellowship.

\textbf{Data availability} Not applicable.

\textbf{Declarations}

\textbf{Conflict of interest } There are no conflicting interests, according to the authors.

\bibliography{sn-bibliography}

\end{document}